\documentclass[11pt,reqno,sumlimits]{amsart}

\usepackage{amssymb, amscd, amsmath, epsfig, mathtools}
\usepackage[utf8]{inputenc}
\usepackage{amsthm}
\usepackage{enumerate}
\usepackage{xcolor}
\usepackage{scalerel}
\usepackage{soul}
\usepackage{tikz-cd}

\usepackage[margin=1.0in]{geometry}

\newtheorem{theorem}{Theorem}[section]
\newtheorem*{theorem*}{Theorem}
\newtheorem{definition}{Definition}[section]
\newtheorem{corollary}{Corollary}[section]

\newtheorem{proposition}{Proposition}[section]
\theoremstyle{definition}
\newtheorem{remark}{Remark}[section]

\newcommand{\R}{\mathbb R}

\newcommand{\calC}{\mathcal C}
\newcommand{\calL}{\mathcal L}
\newcommand{\dvol}{ d\text{Vol}_{g}}

\begin{document}

\title[The Boundary Yamabe Problem General Case]{The Boundary Yamabe Problem, II: General Constant Mean Curvature Case}
\author[J. Xu]{Jie Xu}
\address{
111 Cummington Mall, Department of Mathematics and Statistics, Boston University, Boston, MA, USA, 02215}
\email{xujie@bu.edu}
\address{
600 Dunyu Road, Institute for Theoretical Sciences, Westlake University, Hangzhou, Zhejiang, China, 310030}
\email{xujie67@westlake.edu.cn}

\date{}	
						
\begin{abstract} This article uses the iterative schemes and perturbation methods to completely solve the Han-Li conjecture, i.e. the general boundary Yamabe problem with prescribed constant scalar curvature and constant mean curvature on compact manifolds $ (M, \partial M, g) $ with non-empty smooth boundary, $ \dim M \geqslant 3 $. It is equivalent to show the existence of a real, positive, smooth solution of $ -\frac{4(n -1)}{n - 2} \Delta_{g} u + R_{g} u = \lambda u^{\frac{n+2}{n - 2}} $ in $ M $, $ \frac{\partial u}{\partial \nu} + \frac{n-2}{2} h_{g} u = \frac{n-2}{2} \zeta u^{\frac{n}{n - 2}} $ on $ \partial M $ with some constants $ \lambda, \zeta \in \R $. This boundary Yamabe problem is solved in cases according to the sign of the first eigenvalue $ \eta_{1} $ of the conformal Laplacian with homogeneous Robin condition. The signs of scalar curvature $ R_{g} $ and mean curvature $ h_{g} $ play an important role in this existence result. In contrast to the classical method, the Weyl tensor and classification of boundary points play no role in the proof. In addition, the method is not dimensional specific. The key steps include a new version of monotone iteration scheme for nonlinear elliptic PDEs and nonlinear boundary conditions, an introduction of the perturbed conformal Laplacian, a delicate gluing skill to construct a smooth super-solution and existence of solution of local Yamabe-type equations.
\end{abstract}

\maketitle

\section{Introduction}
This article is a sequel to \cite{XU4}, which solves the boundary Yamabe problem with minimal boundary. The key features and advantages of this local-to-global method are explained in \cite{XU4}. In this article, we are able to show that the same type of local analysis, iteration schemes and perturbation methods developed in \cite{XU4, XU3} can be applied to solve the general boundary Yamabe problem with prescribed constant scalar curvature and nontrivial constant mean curvature on compact manifolds $ (M, \partial M, g) $ with non-empty smooth boundary, with dimensions at least $ 3 $. The Yamabe-type problems are reduced to the analysis of a nonlinear elliptic PDE. Therefore this local-to-global analysis is uniform for solving the Yamabe equations with no boundary condition, with linear boundary condition, and also with non-linear boundary condition. An application of this local-to-global analysis is in the Kazdan-Warner problem, in which we prescribe general scalar curvature functions for metrics within a conformal class, especially for the interesting case: the compact manifolds with positive Yamabe invariants \cite{XU5}.
\medskip

Throughout this article, we denote $ (M, \partial M, g) $ to be a compact manifold with non-empty smooth boundary $ \partial M $, $ n = \dim M \geqslant 3 $. We also denote the interior to be $ M $. Let $ R_{g}, h_{g} $ be the scalar curvature and mean curvature of $ g $, respectively. Set $ a = \frac{4(n- 1)}{n - 2}, p = \frac{2n}{n - 2} $ . Let $ -\Delta_{g} $ be the positive definite Laplace-Beltrami operator. The existence of a metric $ \tilde{g} = u^{p-2} g $ on $ (M, \partial M, g) $ associated with a constant scalar curvature $ \lambda $ and a positive constant mean curvature $ \zeta $ is equivalent to the existence of a real, positive function $ u \in \calC^{\infty}(M) $ satisfying
\begin{equation}\label{intro:eqn1}
\begin{split}
\Box_{g} u & : =  -a\Delta_{g} u + R_{g} u = \lambda u^{p-1} \; {\rm in} \; M; \\
B_{g} u & : = \frac{\partial u}{\partial \nu} + \frac{2}{p-2}  h_{g} u = \frac{2}{p - 2} \zeta u^{\frac{p}{2}} \; {\rm on} \; \partial M.
\end{split}
\end{equation}

The main result of this paper is stated as follows.
\begin{theorem*}
Let $ (M, \partial M, g) $ be a compact manifold with non-empty smooth boundary, $ \dim M \geqslant 3 $. Let $ \eta_{1} $ be the first eigenvalue of the boundary eigenvalue problem $ \Box_{g} = \eta_{1} u $ in $ M $, $ B_{g} u = 0 $ on $ \partial M $. Then: \\
\begin{enumerate}[(i).]
\item If $ \eta_{1} = 0 $, (\ref{intro:eqn1}) admits a real, positive solution $ u \in \calC^{\infty}(M) $ with $ \lambda = \zeta = 0 $;
\item If $ \eta_{1} < 0 $, (\ref{intro:eqn1}) admits a real, positive solution $ u \in \calC^{\infty}(M) $ with some $ \lambda < 0 $ and $ \zeta > 0 $;
\item If $ \eta_{1} > 0 $, (\ref{intro:eqn1}) admits a real, positive solution $ u \in \calC^{\infty}(M) $ with some $ \lambda > 0 $ and $ \zeta > 0 $.
\end{enumerate}
\end{theorem*}
The boundary Yamabe problem for general case is thus completely solved. As in the minimal boundary case \cite{XU4}, the proof of the main theorem relies on the classification of the signs of $ \eta_{1}, R_{g} $ and $ h_{g} $. Case (i) is just the eigenvalue problem; Case (ii) is first solved in Theorem \ref{yamabe:thm2} under the assumption that $ h_{g} > 0 $ everywhere; the general situation when $ \eta_{1} < 0 $ then follows by Theorem \ref{conformal:thm1}; For Case (iii), we prove an existence result for the perturbed boundary Yamabe equation $ -a\Delta_{g} u + \left(R_{g} + \tau \right) u = \lambda_{\tau} u^{p-1} \; {\rm in} \; M $, $ \frac{\partial u}{\partial \nu} + \frac{2}{p-2}  h_{g} u = \frac{2}{p - 2} \zeta u^{\frac{p}{2}} \; {\rm on} \; \partial M $ is given for some specific choice of $ \lambda_{\tau} $, the boundary Yamabe equation with $ R_{g} < 0 $ somewhere and $ h_{g} > 0 $ everywhere is then solved in Theorem \ref{yamabe:thm4} by taking $ \tau \rightarrow 0^{-} $. When $ R_{g} \geqslant 0 $ everywhere, Theorems \ref{conformal:thm1} and \ref{conformal:thm2} are applied to obtain $ R_{\tilde{g}} < 0 $ somewhere and $ h_{\tilde{g}} > 0 $ everywhere for a new metric $ \tilde{g} $ within the conformal class of $ g $. For the importance of the local Yamabe-type equations and the introduction of the perturbation $ \tau $, see Remark \ref{pre:re00}. For the gluing skill and the smoothness of the super-solution, see Remark \ref{yamabe:sre}. For the key step in taking the limit $ \tau \rightarrow 0^{-} $, see Remark \ref{yamabe:sre2}.

Our method is originally inspired by the work of Hintz and Vasy on stability of the (gauged) Einstein equations \cite{Hintz, HVasy}. The connection between the Einstein equation and the Yamabe equation is quite interesting. Roughly speaking, the Einstein equation on a compact manifold is the Euler-Lagrange equation with respect to critical points of the Einstein-Hilbert functional $ \int_{M} R_{g} \dvol $ among all metric of volume $ 1 $, up to a constant. The Yamabe equation is the Euler-Lagrange equation with respect to the critical points of the same functional within a conformal class, after normalizing the volume. The Nash-Moser iteration scheme in the work of Hintz and Vasy inspires us to seek for an appropriate iteration scheme for the elliptic version. We applied other iterative methods in our earlier works \cite{XU2, XU}. We then applied the monotone iteration scheme, which was extended from the local version in \cite{SA} and the closed manifold version in \cite{KW}, to resolve the Yamabe-type equations. The key feature in this article is a new version of the monotone iteration scheme in Theorem \ref{iteration:thm1}, which is designed not only for nonlinear zeroth order terms in the PDE, but also zeroth order nonlinear terms in the boundary conditions.

The history of the Yamabe problem on closed manifolds, starting in 1960, is reviewed in \cite{XU3} and discussed thoroughly in its references. In 1992, Escobar \cite{ESC} initiated the study of the boundary Yamabe problem. Han and Li \cite{HL} conjectured the existence of the positive solution of (\ref{intro:eqn1}) for any constants $ \lambda, \zeta > 0 $. In a recent work of Chen and Sun \cite{CS}, they proved Han-Li conjecture for most cases, in which some restrictions of the geometry of the manifolds were imposed. However, the cases $ n \geqslant 6 $, $ M $ is not locally conformally flat, $ \partial M $ is umbilic, and the Weyl tensor vanishes identically on $ \partial M $ remained open if we consider general compact manifolds with boundary, which are exactly the cases left open in \cite{ESC} for minimal boundary case. In this article, we show the existence of solutions of (\ref{intro:eqn1}) for general $ (M, \partial M, g) $, $ \dim M \geqslant 3 $, associated with $ \zeta \geqslant 0 $ and general $ \lambda \in \R $. Other works in this direction by calculus of variations includes \cite{Brendle}, \cite{Escobar2}, \cite{HL}, etc. in minimal boundary case. In particular, \cite{Marques} among others worked on the case when $ \lambda = 0 $ and $ \tilde{h} $ is a constant. On closed manifolds, \cite{Aubin, PL} provided good survey with classical calculus of variations methods, while a direct analysis can be found in \cite{XU3}.  Another approach to consider Yamabe problem is the application of Yamabe flow for which we refer to \cite{BM}. 

For the crucial difference between the methods use here and the classical methods, we refer to the boundary Yamabe problem with minimal boundary case \cite{XU4}. Because of the nontrivial boundary term, we have to modify several techniques in \cite{XU4}. When $ \eta_{1} > 0 $, our local analysis, iteration schemes and perturbation methods bypass the analysis of $ \lambda(M) < \lambda(\mathbb{S}_{+}^{n}) $, where
\begin{equation*}
\lambda(M) = \inf_{u \neq 0}  \frac{\int_{M} \left( a\lvert \nabla_{g} u \rvert^{2} + R_{g} u^{2} \right) \dvol + \int_{\partial M} \frac{2a}{p-2}h_{g} u^{2} dS}{\left( \int_{M} u^{p} \dvol \right)^{\frac{2}{p}}}
\end{equation*}
and $ \mathbb{S}_{+}^{n} $ is the upper hemisphere of the standard $ n $-sphere. Instead, we construct a sequence of solutions $ \lbrace u_{\tau} \rbrace $ of
\begin{equation}\label{intro:eqn2}
-a\Delta_{g} u_{\tau} + \left(R_{g} + \tau \right) u_{\tau} = \lambda_{\tau} u_{\tau}^{p - 1} \; {\rm in} \; M, \frac{\partial u_{\tau}}{\partial \nu} + \frac{2}{p - 2} h_{g} u_{\tau} = \frac{2}{p-2} \zeta u_{\tau}^{\frac{p}{2}} \; {\rm on} \; \partial M.
\end{equation}
Here $ \tau $ is a small negative parameter. In addition, $ \lambda_{\tau} $ is a perturbation of $ \lambda(M) $,
\begin{equation*}
\lambda_{\tau} = \inf_{u \neq 0} \frac{\int_{M} \left( a\lvert \nabla_{g} u \rvert^{2} + \left(R_{g} + \tau \right) u^{2} \right) \dvol + \int_{\partial M} \frac{2a}{p-2}h_{g} u^{2} dS}{\left( \int_{M} u^{p} \dvol \right)^{\frac{2}{p}}}.
\end{equation*}
The general boundary Yamabe problem is then solved by showing the existence of $ u = \lim_{\tau \rightarrow 0^{-}} u_{\tau} $. The essential difficulty of this approach is to show that for some fixed $ \tau_{0} < 0 $ and $ r > p $, there exist constants $ C_{1}, C_{2} $, independent of $ \tau $,
\begin{equation}\label{intro:eqn3}
\lVert u_{\tau} \rVert_{\calL^{p}(M, g)} \geqslant C_{1} > 0, \lVert u_{\tau} \rVert_{\calL^{r}(M, g)} \leqslant C_{2} < \infty, \forall \tau \in [\tau_{0}, 0).
\end{equation}
As in the minimal boundary case, the upper bound of $ \calL^{r} $-norm in (\ref{intro:eqn3}) relies on the local results in Proposition \ref{pre:prop4} and monotone iteration scheme, especially the construction of super-solutions in Theorem \ref{yamabe:thm3}, which is essentially because of the boundary term $ \frac{2}{p - 2} h_{g} u $ on $ \partial M $. In this article, even the lower bound of $ \calL^{p} $-norm in (\ref{intro:eqn3}) requires the local results and the construction of sub-solutions in Theorem \ref{yamabe:thm3}, since the higher order nonlinear term $ \frac{2}{p-2} \zeta u^{\frac{p}{2}} $ appears in the general constant mean curvature case. In \cite{XU4}, a global $ \calL^{p}$-lower bound was obtained by global analysis. The local result in Proposition \ref{pre:prop4} within an interior domain of $ M $ avoids the analysis near $ \partial M $. The boundary conditions come into play only later in the monotone iteration scheme. Compared with the minimal boundary case in \cite{XU4}, we also need a new version of the global $ \calL^{p} $-type elliptic regularity and a modified monotone iteration scheme in this article, due to the nonlinear term $ \frac{2}{p-2} \zeta u^{\frac{p}{2}} $. In contrast to the classical methods, the local analysis here gives us some flexibility to choose $ \lambda $ and $ \zeta $ in (\ref{intro:eqn1}).
\medskip

This article is organized as follows. In \S2, definitions and essential tools are listed and proved if necessary. In \S3, a result in the perturbation theory for the eigenvalue problem of the conformal Laplacian is proved. In \S4, the monotone iteration scheme is applied to show the existence of the solutions of (\ref{intro:eqn1}) and (\ref{intro:eqn3}), respectively, assuming the existence of corresponding sub-solutions and super-solutions. In \S5, we recall results from \cite{XU4} that a Riemannian metric with general $ R_{g} $ and $ h_{g} $ can be conformally changed to a metric with more tractable scalar curvature and mean curvature; In \S6, the boundary Yamabe problem with nontrivial constant mean curvature is completely solved in the three cases of the main Theorem.
\medskip

\noindent {\bf{Acknowledgement}}: The author would like to thank Prof. Steve Rosenberg and Prof. Gang Tian for their great mentorships in this topic.

\section{The Preliminaries}
In this section, necessary definitions, setups and required tools for the main result are listed. In particular, required spaces and associated Sobolev spaces are defined; results of $ W^{s, p} $-type and $ H^{s} $-type elliptic regularities are listed; maximum principles and Sobolev embeddings are given; existence and uniqueness of solutions of some second order elliptic PDE with inhomogeneous Robin conditions are shown. Throughout this section, we consider the spaces with dimensions no less than $ 3 $.

\begin{definition}\label{pre:def1}
We say $ (\Omega, g) $ being a Riemannian domain if (i)$ \Omega $ is a connected, bounded, open subset of $ \R^{n} $ with smooth boundary $ \partial \Omega $ equipped with some Riemannian metric $ g $ that can be extended smoothly to $ \bar{\Omega} $; (ii) in addition, $ (\bar{\Omega}, g) $ must be a compact Riemannian manifold with boundary $ \partial \Omega $.
\end{definition}
\medskip

We define integer-ordered Sobolev spaces on compact manifolds with smooth boundary $ (M, \partial M, g) $, and on Riemannian domain. In this article, we always define the inclusion map
\begin{equation*}
\imath: \partial M \hookrightarrow M
\end{equation*}
from the boundary to the whole manifold, and thus $ \partial M $ admits a canonical Riemannian metric $ \imath^{*} g $.
\begin{definition}\label{pre:def2} Let $ (M, \partial M, g) $ be a compact Riemannian manifold with smooth boundary, let $ \dim M = n $. Let $ d\omega $ be the Riemannian density with local expression $ \dvol $. Let $ dS $ be the induced boundary density on $ \partial M $. For real valued functions $ u $, we set: 

(i) 
For $1 \leqslant p < \infty $, 
\begin{align*}
\mathcal{L}^{p}(M, g)\ &{\rm is\ the\ completion\ of}\ \left\{ u \in \calC_{c}^{\infty}(M) : \Vert u\Vert_{p,g}^p :=\int_{M} \left\lvert u \right\rvert^{p} d\omega < \infty \right\}; \\
\mathcal{L}^{p}(\Omega, g)\ &{\rm is\ the\ completion\ of}\ \left\{ u \in \calC_c^{\infty}(\Omega) : \Vert u\Vert_{p,g}^p :=\int_{\Omega} \left\lvert u \right\rvert^{p} d\text{Vol}_{g} < \infty \right\}.
\end{align*}

(ii)
For $\nabla$ the Levi-Civita connection of $g$, and for 
$ u \in \calC^{\infty}(M) $,
\begin{equation}\label{pre:eqn1}
\lvert \nabla^{k} u \rvert_g^{2} := (\nabla^{\alpha_{1}} \dotso \nabla^{\alpha_{k}}u)( \nabla_{\alpha_{1}} \dotso \nabla_{\alpha_{k}} u).
\end{equation}

In particular, $ \lvert \nabla^{0} u \rvert^{2}_g = \lvert u \rvert^{2}_g $ and $ \lvert \nabla^{1} u \rvert^{2}_g = \lvert \nabla u \rvert_{g}^{2} $.

(iii) For $ s \in \mathbb{N}, 1 \leqslant p < \infty $,
\begin{equation}\label{pre:eqn2}
\begin{split}
W^{s, p}(M, g) & = \left\{ u \in \mathcal{L}^{p}(M, g) : \lVert u \rVert_{W^{s, p}(M, g)}^{p} = \sum_{j=0}^{s} \int_{M} \left\lvert \nabla^{j} u \right\rvert^{p}_g d\omega < \infty \right\}; \\
W^{s, p}(\Omega, g) &= \left\{ u \in \mathcal{L}^{p}(\Omega, g) : \lVert u \rVert_{W^{s, p}(\Omega, g)}^{p} = \sum_{j=0}^{s} \int_{\Omega} \left\lvert \nabla^{j} u \right\rvert^{p}_g d\text{Vol}_{g} < \infty \right\}.
\end{split}
\end{equation}

Similarly, $ W_{0}^{s, p}(M, g) $ is the completion of $ \calC_{c}^{\infty}(M) $ with respect to the $ W^{s, p} $-norm. In particular, $ H^{s}(M, g) : = W^{s, 2}(M, g), s \in \mathbb{N}, 1 \leqslant p' < \infty $ are the usual Sobolev spaces, and we similarly define $H_{0}^{s}(M, g) $, $ W_{0}^{s, p}(\Omega, g) $ and $ H_{0}^{s}(\Omega, g) $.

(iv) With an open cover $ \lbrace U_{\xi}, \phi_{\xi} \rbrace $ of $ (M, \partial M, g) $ and a smooth partition of unity $ \lbrace \chi_{\xi} \rbrace $ subordinate to this cover, we can define the $ W^{s, p} $-norm locally, which is equivalent to the definition above.
\begin{equation*}
\lVert u \rVert_{W^{s, p}(M, g)} = \sum_{\xi} \lVert \left(\phi_{\xi}^{-1}\right)^{*} \chi_{\xi} u \rVert_{W^{s, p}(\phi_{\xi}(U_{\xi}), g)}.
\end{equation*}
\end{definition}
\medskip

A local $ L^{p} $ regularity is required for some type of Robin boundary condition, due to Agmon, Douglis, and Nirenberg \cite{Niren4}. It is worth mentioning that the result in Proposition \ref{pre:prop1} below holds when $ \Omega $ in particular is a $ n $ dimensional hemisphere denoted by $ \sum_{i = 1}^{n - 1} x_{i}^{2} + t^{2} \leqslant 1, t \geqslant 0 $ where the boundary condition is only defined on $ t = 0 $ and $ u $ in (\ref{pre:eqn3}) vanishes outside the hemisphere, see \cite[Thm.~15.1]{Niren4}. The Schauder estimates holds in the same manner, see \cite[Thm.~7.1, Thm.~7.2]{Niren4}. This is particularly useful since for the global analysis in next theorem, we will choose a cover of $ (M, \partial M, g) $, and for any boundary chart $ (U, \phi) $ of $ (M, \partial M, g) $, the intersection $ \phi(\bar{U} \cap M) $ is a one-to-one correspondence to a hemisphere, provided that $ \partial M $ is smooth enough. It resolves the issue for the boundary charts.

\begin{proposition}\label{pre:prop1}\cite[Thm.~7.3, Thm.~15.2]{Niren4} Let $ (\Omega, g) $ be a Riemannian domain where the boundary $ \partial \Omega $ satisfies Lipschitz condition. Let $ \nu $ be the outward unit normal vector along $ \partial \Omega $. Let $ L $ be the second order elliptic operator on $ \Omega $ with smooth coefficients up to $ \partial M $ and $ f \in \calL^{p}(\Omega, g) $, $ f' \in W^{1, p}(\Omega, g) $ for some $ p \in (1, \infty) $. Let $ u \in H^{1}(\Omega, g) $ be the weak solution of the following boundary value problem
\begin{equation}\label{pre:eqn3}
L u = f \; {\rm in} \; \Omega, Bu : = \frac{\partial u}{\partial \nu} + c(x) u = f' \; {\rm on} \; \partial \Omega,
\end{equation}
where $ c \in \calC^{\infty}(\partial \Omega) $. Then $ u \in W^{2, p}(\Omega, g) $ and the following estimates holds provided $ u \in \calL^{p}(\Omega, g) $:
\begin{equation}\label{pre:eqn4}
\lVert u \rVert_{W^{2, p}(\Omega, g)} \leqslant C^{*} \left( \lVert Lu \rVert_{\calL^{p}(\Omega, g)} + \lVert Bu \rVert_{W^{1, p}(\Omega, g)} +  \lVert u \rVert_{\calL^{p}(\Omega, g)} \right).
\end{equation}
Here the constant $ C^{*} $ depends on $ L, p $ and $ (\Omega, g) $.
\end{proposition}
\medskip

\begin{theorem}\label{pre:thm1} Let $ (M, \partial M, g) $ be a compact manifold with non-empty smooth boundary. Let $ \nu $ be the unit outward normal vector along $ \partial M $. Let $ L $ be a uniform second order elliptic operator on $ M $ with smooth coefficients up to $ \partial M $. Let $ f \in \calL^{p}(M, g), \tilde{f} \in W^{1, p}(M, g) $. Let $ u \in H^{1}(M, g) $ be a weak solution of the following boundary value problem
\begin{equation}\label{pre:eqn5}
L u = f \; {\rm in} \; M, Bu : = \frac{\partial u}{\partial \nu} + c(x) u = \tilde{f} \; {\rm on} \; \partial M.
\end{equation}
Here $ c \in \calC^{\infty}(M) $. If, in addition, $ u \in \calL^{p}(M, g) $, then $ u \in W^{2, p}(M, g) $ with the following estimates
\begin{equation}\label{pre:eqn6}
\lVert u \rVert_{W^{2, p}(M, g)} \leqslant C \left( \lVert Lu \rVert_{\calL^{p}(M, g)} + \lVert Bu \rVert_{W^{1, p}(M, g)} + \lVert u \rVert_{\calL^{p}(M, g)} \right).
\end{equation}
Here $ C $ depends on $ L, p, c $ and the manifold $ (M, \partial M, g) $, and is independent of $ u $.
\end{theorem}
\begin{proof} This proof is essentially the same as in Theorem 3.1 of \cite{XU4}. The only difference is to handle the inhomonegeous Robin boundary condition. Choose a finite cover of $ (M, \partial M, g) $, say
\begin{equation*}
(M, \partial M, g) = \left( \bigcup_{\alpha} (U_{\alpha}, \phi_{\alpha}) \right) \cup \left( \bigcup_{\beta} (U_{\beta}, \phi_{\beta}) \right)
\end{equation*}
where $ \lbrace U_{\alpha}, \phi_{\alpha} \rbrace $ are interior charts and $ \lbrace U_{\beta}, \phi_{\beta} \rbrace $ are boundary charts. Choose a partition of unity $ \lbrace \chi_{\alpha}, \chi_{\beta} \rbrace $ subordinate to this cover, where $ \lbrace \chi_{\alpha} \rbrace $ are associated with interior charts and $ \lbrace \chi_{\beta} \rbrace $ are associated with boundary charts. The local expression of the differential operator for interior charts is of the form
\begin{equation*}
L \mapsto \left( \phi_{\alpha}^{-1} \right)^{*} L \phi_{\alpha}^{*} : \calC^{\infty}(\phi_{\alpha}(U_{\alpha})) \rightarrow \calC^{\infty}(\phi_{\alpha}(U_{\alpha})) 
\end{equation*}
which can be extended to Sobolev spaces with appropriate orders. The same expression applies for boundary charts. Denote
\begin{align*}
L_{\alpha} & =  \left( \phi_{\alpha}^{-1} \right)^{*} L \phi_{\alpha}^{*}, L_{\beta} =  \left( \phi_{\beta}^{-1} \right)^{*} L \phi_{\beta}^{*}; \frac{\partial}{\partial \nu'} =  \left( \phi_{\beta}^{-1} \right)^{*} \left( \frac{\partial}{\partial \nu} \right)\phi_{\beta}^{*} \\
\left(\phi_{\alpha}^{-1} \right)^{*} \chi_{\alpha} & = \chi_{\alpha}', \left(\phi_{\beta}^{-1} \right)^{*} \chi_{\beta} = \chi_{\beta}'; \left(\phi_{\beta}^{-1} \right)^{*} c = c_{\beta}'; \\
\left(\phi_{\alpha}^{-1} \right)^{*} u & = u_{\alpha}', \left(\phi_{\beta}^{-1} \right)^{*} u = u_{\beta}', \left(\phi_{\alpha}^{-1} \right)^{*} f = f_{\alpha}', \left(\phi_{\beta}^{-1} \right)^{*} f = f_{\beta}', \left(\phi_{\beta}^{-1} \right)^{*} \tilde{f} = \tilde{f}_{\beta}'
\end{align*}
With these notations, the local expressions of our PDE with respect to $ u_{\alpha}', u_{\beta}' $ associated with (\ref{pre:eqn5}) in each chart, respectively, are as follows:
\begin{equation}\label{pre:eqn7}
\begin{split}
L_{\alpha} \left( \chi_{\alpha}' u_{\alpha}' \right) - [L_{\alpha}, \chi_{\alpha}']u_{\alpha}'  & = \chi_{\alpha}' f_{\alpha}' \; {\rm in} \; \phi_{\alpha}(U_{\alpha}), \chi_{\alpha}' u_{\alpha}'  = 0 \; {\rm on} \;\partial \phi_{\alpha}(U_{\alpha}); \\
L_{\beta} \left( \chi_{\beta}' u_{\beta}' \right) - [L_{\beta}, \chi_{\beta}']u_{\beta}'  & = \chi_{\beta}' f_{\beta}' \; {\rm in} \; \phi_{\beta}(U_{\beta}), \\
\frac{\partial \chi_{\beta}' u_{\beta}' }{\partial \nu'} + c_{\beta}' \chi_{\beta}' u_{\beta}' - \frac{\partial \chi_{\beta}'}{\partial \nu'} u_{\beta}' & = \chi_{\beta}' \tilde{f}_{\beta} \; {\rm on} \; \partial \phi_{\beta}(\bar{U}_{\beta} \cap M), \chi_{\beta}' u_{\beta}' = 0 \; {\rm on} \; \partial \phi_{\beta}(U_{\beta}) \backslash \left( \partial \phi_{\beta}(\bar{U}_{\beta} \cap M) \right).
\end{split}
\end{equation}
Here $ [L, \chi] $ is a commutator defined as
\begin{equation*}
[L, \chi] u = L(\chi u) - \chi (Lu).
\end{equation*}
Applying local estimates in Proposition \ref{pre:prop1} with the extra boundary term $ \chi_{\beta}' \tilde{f}_{\beta} $, and then glue estimates in each chart together, we conclude that (\ref{pre:eqn6}) holds. The argument is exactly the same as in \cite[Thm.~3.1]{XU4}.
\end{proof}
\medskip

When $ L $ is injective, we show in next theorem that $ \lVert u \rVert_{\calL^{p}(M, g)} $ can be absorbed by $ \lVert Lu \rVert_{\calL^{p}(M, g)} $.
\begin{theorem}\label{pre:thm2} Let $ (M, \partial M, g) $ be a compact manifold with non-empty smooth boundary. Let $ \nu $ be the unit outward normal vector along $ \partial M $ and $ p > \dim M $. Let $ L: \calC^{\infty}(M) \rightarrow \calC^{\infty}(M) $ be a uniform second order elliptic operator on $ M $ with smooth coefficients up to $ \partial M $ and can be extended to $ L : W^{2, p}(M, g) \rightarrow \calL^{p}(M, g) $. Let $ f \in \calL^{p}(M, g), \tilde{f} \in W^{1, p}(M, g) $. Let $ u \in H^{1}(M, g) $ be a weak solution of the following boundary value problem
\begin{equation}\label{pre:eqn8}
L u = f \; {\rm in} \; M, Bu = \frac{\partial u}{\partial \nu} + c(x) u = \tilde{f} \; {\rm on} \; \partial M.
\end{equation}
Here $ c \in \calC^{\infty}(M) $. Assume also that $ \text{Ker}(L) = \lbrace 0 \rbrace $ associated with the homogeneous Robin boundary condition. If, in addition, $ u \in \calL^{p}(M, g) $, then $ u \in W^{2, p}(M, g) $ with the following estimates
\begin{equation}\label{pre:eqn9}
\lVert u \rVert_{W^{2, p}(M, g)} \leqslant C' \left(\lVert Lu \rVert_{\calL^{p}(M, g)} + \lVert Bu \rVert_{W^{1, p}(M, g)} \right)
\end{equation}
Here $ C' $ depends on $ L, p, c $ and the manifold $ (M, \partial M, g) $, and is independent of $ u $.
\end{theorem}
\begin{proof}
It is enough to show that there exists a constant $ E $, independent of $ u $, such that
\begin{equation}\label{pre:eqn10}
\lVert u \rVert_{\calL^{p}(M, g)} \leqslant E \left( \lVert Lu \rVert_{\calL^{p}(M, g)} + \lVert Bu \rVert_{W^{1, p}(M, g)}\right)
\end{equation}
provided that $ L $ is injective. The inequality (\ref{pre:eqn10}) is proven by contradiction, exactly the same as \cite[Thm.~3.2]{XU4}.
\end{proof}
\medskip

Maximum principles play a central role in nontriviality of the solution of (\ref{intro:eqn1}), which are stated below in local versions.
\begin{theorem}\label{pre:thm3} Let $ (\Omega, g) $ be a Riemannian domain with $ \dim \Omega \geqslant 3 $.
(i) \cite[Cor.~3.2]{GT} (Weak Maximum Principle) Let $ L $ be a second order elliptic operator of the form
\begin{equation*}
Lu = -\sum_{\lvert \alpha \rvert = 2} -a_{\alpha}(x) \partial^{\alpha} u + \sum_{\lvert \beta \rvert = 1} -b_{\beta}(x) \partial^\beta u + c(x) u
\end{equation*}
\noindent where $ a_{\alpha}, b_{\beta}, c \in \calC^{\infty}(\Omega) $ are smooth and bounded real-valued functions on the bounded domain $ \Omega \subset \R^{n} $. Let $ u \in \calC^2(\bar{\Omega}) $. Denote $u^{-} : = \min(u, 0)$, and $ u^{+} : = \max(u, 0) $. Then we have
\begin{equation}\label{pre:eqn11}
\begin{split}
Lu \geqslant 0, c(x) \geqslant 0 & \Rightarrow \inf_{\Omega} u = \inf_{\partial \Omega} u^{-}; \\
Lu \leqslant 0, c(x) \geqslant 0 & \Rightarrow \sup_{\Omega} u = \sup_{\partial \Omega} u^{+}.
\end{split}
\end{equation}

(ii) \cite[Thm.~3.1]{GT} Let $ u $, $ L $ and its coefficients be the same as in (i) above. Then we have
\begin{equation}\label{pre:eqn12}
\begin{split}
Lu \geqslant 0, c(x) = 0 & \Rightarrow \inf_{\Omega} u = \inf_{\partial \Omega} u; \\
Lu \leqslant 0, c(x) = 0 & \Rightarrow \sup_{\Omega} u = \sup_{\partial \Omega} u.
\end{split}
\end{equation}

(iii) \cite[Thm.~3.5]{GT} (Strong Maximum Principle) Let $ u \in \calC^{2}(\Omega) $. Assume that $\partial \Omega \in \calC^{\infty}$. Let $ L $ be a second order uniformly elliptic operator as above. If $Lu \geqslant 0$,  $ c(x) \geqslant 0 $, and if $ u $ attains a nonpositive minimum over $ \bar{\Omega} $ in an interior point, then $ u $ is constant within $ \Omega $; If $ Lu \leqslant 0 $, $ c(x) \geqslant 0 $, and if $ u $ attains a nonnegative maximum over $ \bar{\Omega} $ in an interior point, then $ u $ is constant within $ \Omega $.

(iv) \cite[Ch.~8]{GT} All weak and strong maximum principles above hold when $ u \in H^{1}(\Omega, g) $ or $ u \in H^{1}(M, g) $, respectively, provided that $ L $ is uniformly elliptic with some other restrictions.
\end{theorem}
\medskip

Sobolev embeddings is critical for the regularity of solution of (\ref{intro:eqn1}).
\begin{proposition}\label{pre:prop2}\cite[Ch.~2]{Aubin}  (Sobolev Embeddings) 
Let $ (M, \partial M, g) $ be a compact manifold with non-empty smooth boundary $ \partial M $.

(i) For $ s \in \mathbb{N} $ and $ 1 \leqslant p \leqslant p' < \infty $ such that
\begin{equation}\label{pre:eqns1}
   \frac{1}{p} - \frac{s}{n} \leqslant \frac{1}{p'},
\end{equation}
\noindent  $ W^{s, p}(M, g) $ continuously embeds into $ \mathcal{L}^{p'}(M, g) $ with the following estimates: 
\begin{equation}\label{pre:eqns2}
\lVert u \rVert_{\calL^{p'}(M, g)} \leqslant K \lVert u \rVert_{W^{s, p}(M, g)}.
\end{equation}

(ii) For $ s \in \mathbb{N} $, $ 1 \leqslant p < \infty $ and $ 0 < \alpha < 1 $ such that
\begin{equation}\label{pre:eqns3}
  \frac{1}{p} - \frac{s}{n} \leqslant -\frac{\alpha}{n},
\end{equation}
Then  $ W^{s, p}(M, g) $ continuously embeds in the H\"older space $ \calC^{0, \alpha}(M) $ with the following estimates:
\begin{equation}\label{pre:eqns4}
\lVert u \rVert_{\calC^{0, \alpha}(M)} \leqslant K' \lVert u \rVert_{W^{s, p}(M, g)}.
\end{equation}

(iii) Both embeddings above are compact embeddings provided that the strict inequalities hold in (\ref{pre:eqns1}) and (\ref{pre:eqns3}), respectively.
\end{proposition} 
\medskip

The solvability of linear PDE with Robin boundary condition requires the trace theorem below.
\begin{proposition}\label{pre:prop3}\cite[Prop.~4.5]{T}
Let $ (M, \partial M, g) $ be a compact manifold with non-empty smooth boundary. Let $ u \in H^{1}(M, g) $. Then there exists a bounded linear operator
\begin{equation*}
T : H^{1}(M, g) \rightarrow \calL^{2}(\partial M, \imath^{*}g)
\end{equation*}
such that
\begin{equation}\label{pre:eqn13}
\begin{split}
T u & = u \bigg|_{\partial M}, \; \text{if} \; u \in \calC^{\infty}(M) \cap H^{1}(M, g); \\
\lVert T u \rVert_{\calL^{2}(\partial M, \imath^{*} g)} & \leqslant K'' \lVert u \rVert_{H^{1}(M, g)}.
\end{split}
\end{equation}
Here $ K'' $ only depends on $ (M, g) $ and is independent of $ u $. Furthermore, the map $ T : H^{1}(M, g) \rightarrow H^{\frac{1}{2}}(\partial M, \imath^{*} g) $ is surjective.
\end{proposition}
\medskip

We consider the existence and uniqueness of the solutions of the following PDE
\begin{equation}\label{pre:eqn14}
-a\Delta_{g} u + Au = f \; {\rm in} \; M, \frac{\partial u}{\partial \nu} + c(x) u = \tilde{f} \; {\rm on} \; \partial M
\end{equation}
with appropriate choices of constants $ a, A $ and functions $ f, \tilde{f}, c $. Pairing any test function $ v \in \calC_{c}^{\infty}(M) $ on both sides of (\ref{pre:eqn13}), we have
\begin{equation*}
\int_{M} \left( a \nabla_{g} u \cdot \nabla_{g} v + Auv \right) d\omega + \int_{\partial M} c uv dS = \int_{M} fv d\omega + \int_{\partial M} \tilde{f} v dS.
\end{equation*}
Denote
\begin{equation}\label{pre:eqn15}
\tilde{L} v : = \int_{M} fv d\omega + \int_{\partial M} \tilde{f} v dS.
\end{equation}
We conclude immediately that $ \tilde{L} : H^{1}(M, g) \rightarrow \R $ is a bounded linear operator, provided that $ f \in \calL^{2}(M, g) $ and $ \tilde{f} \in H^{1}(M, g) $. It is natural to consider the solvability of (\ref{pre:eqn14}) in the weak sense.
\begin{definition}\label{pre:def3}
Let $ f \in \calL^{2}(M, g) $ and $ \tilde{f} \in H^{1}(M, g) $. Let $ \tilde{L} $ be defined as in (\ref{pre:eqn15}). Define
\begin{equation}\label{pre:eqn16}
B[u, v] = \int_{M} \left( a \nabla_{g} u \cdot \nabla_{g} v + Auv \right) d\omega + \int_{\partial M} c uv dS.
\end{equation}
We say that $ u \in H^{1}(M, g) $ is a weak solution of (\ref{pre:eqn13}) if and only if
\begin{equation*}
B[u, v] = \tilde{L} v, \forall v \in H^{1}(M, g).
\end{equation*}
\end{definition}
Next theorem provides the existence and uniqueness of the solution of (\ref{pre:eqn13}) and the injectivity of the operator $ -a\Delta_{g} + A $. The key is the assumption that $ c > 0 $ everywhere on $ \partial M $.
\begin{theorem}\label{pre:thm4}
Let $ (M, \partial M, g) $ be a compact manifold with non-empty smooth boundary. Let $ \nu $ be the unit outward normal vector along $ \partial M $. Let $ a, A > 0 $ be any positive constants. Let $ f \in \calL^{2}(M, g), \tilde{f} \in H^{1}(M, g) $ and $ c \in \calC^{\infty}(M) $ with $ c > 0 $ on $ \partial M $. Then the PDE (\ref{pre:eqn14}) has a unique weak solution $ u \in H^{1}(M, g) $ in the sense of Definition \ref{pre:def3}. Furthermore, $ -a\Delta_{g} + A $ is injective with homogeneous boundary condition
\begin{equation*}
\frac{\partial u}{\partial \nu} + c(x) u = 0 \; {\rm on} \; \partial M.
\end{equation*}
\end{theorem}
\begin{proof} Due to Definition \ref{pre:def3}, we show the existence of a unique $ u \in H^{1}(M, g) $ such that
\begin{equation*}
B[u, v] = \tilde{L} v, \forall v \in H^{1}(M, g).
\end{equation*}
The linear operator $  \tilde{L}  : H^{1}(M, g) \rightarrow \R $ is continuous, since
\begin{align*}
\lvert  \tilde{L} v \rvert & \leqslant \int_{M} \lvert f \rvert \lvert v \rvert d\omega + \int_{\partial M} \lvert \tilde{f} \rvert \lvert v \rvert dS \leqslant \lVert f \rVert_{\calL^{2}(M, g)} \lVert v \rVert_{\calL^{2}(M, g)} + \lVert \tilde{f} \rVert_{\calL^{2}(\partial M, \imath^{*} g)} \lVert v\rVert_{\calL^{2}(\partial M, \imath^{*} g)} \\
& \leqslant \left( \lVert f \rVert_{\calL^{2}(M, g)} + \left( K'' \right)^{2} \lVert \tilde{f} \rVert_{H^{1}(M, g)} \right) \lVert v \rVert_{H^{1}(M, g)}.
\end{align*}
According to the same argument in \cite[Thm.~4.2]{XU4}, we conclude by Lax-Milgram \cite{Lax} that (\ref{pre:eqn13}) has a unique weak solution $ u \in H^{1}(M, g) $. The injectivity of the operator $ -a\Delta_{g} + A $ follows exactly the same as in Theorem 4.2 of \cite{XU4}.
\end{proof}
\medskip

When the first eigenvalue $ \eta_{1} $ of conformal Laplacian with Robin condition is positive, we need to consider a perturbed Yamabe equation first:
\begin{equation*}
-a\Delta_{g} u + \left(R_{g} + \tau \right) u = \lambda u^{p-1} \; {\rm in} \; M, B_{g} u = \frac{2}{p-2} \zeta u^{\frac{p}{2}} \; {\rm on} \; \partial M
\end{equation*}
for some negative constant $ \tau < 0 $. The following result, which is a local version of perturbed Yamabe equation with trivial Dirichlet boundary condition, plays a central role in this boundary Yamabe problem. We proved this result in \cite{XU3}, and applied this result to proof the Yamabe problem on closed manifolds and the boundary Yamabe problem with minimal boundary case.
\begin{proposition}\label{pre:prop4}\cite[Prop.~3.3]{XU3}
Let $ (\Omega, g) $ be Riemannian domain in $\R^n$, $ n \geqslant 3 $, with $C^{\infty} $ boundary, and with ${\rm Vol}_g(\Omega)$ and the Euclidean diameter of $\Omega$ sufficiently small. Let $ \tau < 0 $ be any negative constant. Assume $ R_{g} < 0 $ everywhere within the small enough closed domain $ \bar{\Omega} $. Then for any $ \lambda > 0 $ the following Dirichlet problem
\begin{equation}\label{pre:eqn17}
-a\Delta_{g} u + \left( R_{g} + \tau \right) u = \lambda u^{p-1} \; {\rm in} \; \Omega, u = 0 \; {\rm on} \; \partial \Omega.
\end{equation}
has a real, positive solution $ u \in \calC^{\infty}(\Omega) \cap \calC^{1, \alpha}(\bar{\Omega}) $ vanishes at $ \partial \Omega $. 
\end{proposition}
\begin{remark}\label{pre:re00}
Let $ \lambda_{1} $ be the first nonzero eigenvalue of $ -\Delta_{g} $ on Riemannian domain $ (\Omega, g) $ with Dirichlet boundary condition. Recall that in Proposition 3.3 of \cite{XU3}, the smallness of $ \Omega $ is determined by
\begin{equation}\label{pre:eqn18}
\sup_{x \in M} \lvert R_{g} \rvert + \lvert \tau \rvert \leqslant a \lambda_{1}, \frac{a}{n} - \left( \frac{n - 2}{2n} + \frac{1}{2} \right) \left( \sup_{x \in M} \lvert R_{g} \rvert + \lvert \tau \rvert \right) \lambda_{1}^{-1} \geqslant 0.
\end{equation}
(\ref{pre:eqn18}) will be used in Section 6.

We point out that the introduction of $ \tau $ here gives the solvability of (\ref{pre:eqn17}) for all $ n \geqslant 3 $. In addition, the solvability is regardless of the vanishing/non-vanishing of the Weyl tensor since the perturbed conformal Laplacian is no longer a conformal invariance, and thus different choices of local curvature conditions play important roles for the existence of the solution of (\ref{pre:eqn17}).
\end{remark}

\section{Yamabe Invariants and Eigenvalue Problem of Conformal Laplacian}
In this section, we discuss the first eigenvalue of conformal Laplacian $ \Box_{g} $ with oblique boundary condition $ B_{g} $. In order to construct super-solutions for later sections, we discuss an eigenvalue problem of conformal Laplacian with a slightly different boundary condition.

Consider the eigenvalue problem of conformal Laplacian $ \Box_{g} $ with boundary condition $ B_{g} \varphi = 0 $
\begin{equation}\label{eigen:eqn1}
-a\Delta_{g} \varphi + R_{g} \varphi = \eta_{1} \varphi \; {\rm in} \; M, \frac{\partial \varphi}{\partial \nu} + \frac{2}{p-2} h_{g} \varphi = 0 \; {\rm on} \; \partial M.
\end{equation}
It is well known that (\ref{eigen:eqn1}) admits a real, positive solution $ \varphi \in \calC^{\infty}(M) $ with $ \eta_{1} $ the smallest nonzero eigenvalue.The first result is to show that the first eigenvalue $ \eta_{1} $ of $ \Box_{g} $ with homogenous Robin condition is a conformal invariant, due to Escobar \cite{ESC}.
\begin{proposition}\label{eigen:prop1}\cite[Prop.~1.3.]{ESC}
Let $ \tilde{g} = u^{p-2} g $ be a conformal metric to $ g $. Let $ \eta_{1} $ and $ \tilde{\eta}_{1} $ be the first eigenvalue of $ \Box_{g} $ and $ \Box_{\tilde{g}} $ with boundary conditions $ B_{g} = 0 $ and $ B_{\tilde{g}} = 0 $, respectively. Then either the signs of $ \eta_{1} $ and $ \tilde{\eta}_{1} $ are the same or $ \eta_{1} = \tilde{\eta}_{1} = 0 $.
\end{proposition}
\medskip

We consider the following slightly different eigenvalue problem
\begin{equation}\label{eigen:eqn2}
-a\Delta_{g} \phi + R_{g} \phi = \eta_{1, \beta} \phi \; {\rm in} \; M, \frac{\partial \phi}{\partial \nu} + \left(\frac{2}{p-2}h_{g} - \beta \right) \phi = 0 \; {\rm on} \; \partial M
\end{equation}
for some constant $ \beta > 0 $.
However, $ \eta_{1}' $ is not the first eigenvalue of conformal Laplacian, and hence is not a conformal invariant. Next theorem shows the existence of the solution of (\ref{eigen:eqn2}), and the choices of $ \beta $ so that $ \eta_{1, \beta} > 0 $ whenever $ \eta_{1} > 0 $ for some metric $ g $.
\begin{theorem}\label{eigen:thm1}
Let $ (M, \partial M, g) $ be a compact manifold with non-empty smooth boundary. Let $ \nu $ be the unit outward normal vector along $ \partial M $. If the first eigenvalue $ \eta_{1} $ of conformal Laplacian $ \Box_{g} $ and the mean curvature $ h_{g} $ are positive, then there exists some constant $ C_{\eta_{1}} > 0 $ such that for all $ \beta \in [0, C_{\eta_{1}}] $, the eigenvalue problem (\ref{eigen:eqn2}) has a real, positive solution $ \phi \in \calC^{\infty}(M) $ with some $ \eta_{1, \beta} > 0 $ depending on $ \beta $.
\end{theorem}
\begin{proof} For any choice of $ \beta $, the existence of the solution (\ref{eigen:eqn1}) is equivalent to the existence of the minimizer of
\begin{equation}\label{eigen:eqn3}
I_{\beta}[u] = \frac{\int_{M} a \lvert \nabla_{g} u \rvert^{2} d\omega + \int_{M} R_{g} u^{2} d\omega + \int_{\partial M} (\frac{2a}{p-2}h_{g} - \beta) u^{2} dS}{\int_{M} u^{2} d\omega}.
\end{equation}
We can discuss the minimizer of $ I[u] $ within the following admissible set
\begin{equation}\label{eigen:eqn4}
\mathcal{A} = \lbrace u \in H^{1}(M, g) : \lVert u \rVert_{\calL^{2}(M, g)} = 1 \rbrace.
\end{equation}
By the same argument as the standard eigenvalue problem, we conclude that (\ref{eigen:eqn2}) admits a solution $ \phi \in H^{1}(M, g) $. The existence part is due to the uniform ellipticity of $ -\Delta_{g} $ and coercivity of $ I_{\beta}[u] $. The facts \cite[Prop.~4.4, 4.5]{T} that the inclusion map $ H^{s + \sigma}(M, g) \hookrightarrow H^{s}(M, g) $ for all $ s \geqslant 0, \sigma > 0 $ is compact and the trace operator $ \tau $ is a bounded linear operator $ \tau: H^{s}(M, g) \rightarrow H^{s - \frac{1}{2}}(\partial M, \imath^{*}g) $ for $ s > \frac{1}{2} $ implies that $ T: H^{1}(M, g) \rightarrow \calL^{2}(\partial M, g) $ defined in Proposition \ref{pre:prop3} is compact. Hence the existence of the minimizer of (\ref{eigen:eqn3}) is proved.

Furthermore, we see that $ \phi \geqslant 0 $ since $ I[u] = I[\lvert u \rvert] $. By elliptic regularity, we conclude that $ \phi \in \calC^{\infty}(M) $. Maximum principles in Theorem \ref{pre:thm3} implies that $ \phi > 0 $ in $ M $ and thus can be extended to $ M $ positively. 
\medskip

We now show that if $ \eta_{1} > 0 $ in (\ref{eigen:eqn1}), there exists some $ C_{\eta_{1}} $ such that (\ref{eigen:eqn2}) admits a positive solution $ \phi $ associated with $ \eta_{1, \beta} $ for all $ \beta \in [0, C_{\eta_{1}}] $. Observe that the first eigenvalue $ \eta_{1} $ and $ \eta_{1, \beta} $ are characterized by
\begin{equation*}
\eta_{1} = I[\varphi] = \inf_{u \in \mathcal{A}} I_{0}[u], \eta_{1, \beta} = I[\phi] = \inf_{u \in \mathcal{A}} I_{\beta}[u]
\end{equation*}
after normalizing both $ \varphi $ and $ \phi $.The first characterization implies that
\begin{equation*}
0 < \eta_{1} \leqslant \int_{M} a \lvert \nabla_{g} u \rvert^{2} d\omega + \int_{M} R_{g} u^{2} d\omega + \int_{\partial M} \frac{2a}{p-2}h_{g} u^{2} dS, \forall u \in \mathcal{A}.
\end{equation*}
Using this and the trace theorem in Proposition \ref{pre:prop3}, the quantity $ I_{\beta}[u] $ for all elements $ u \in \mathcal{A} $ and small enough $ \beta $ satisfies
\begin{align*}
I_{\beta}[u] & = \int_{M} a \lvert \nabla_{g} u \rvert^{2} d\omega + \int_{M} R_{g} u^{2} d\omega + \int_{\partial M} (\frac{2a}{p-2}h_{g} - \beta) u^{2} dS \\
& \geqslant \int_{M} a \lvert \nabla_{g} u \rvert^{2} d\omega + \int_{M} R_{g} u^{2} d\omega + \int_{\partial M} \frac{2a}{p-2}h_{g}  u^{2} dS - \beta \lVert u \rVert_{\calL^{2}(\partial M, \imath^{*} g)}^{2} \\
& \geqslant  \int_{M} a \lvert \nabla_{g} u \rvert^{2} d\omega + \int_{M} R_{g} u^{2} d\omega + \int_{\partial M} \frac{2a}{p-2}h_{g}  u^{2} dS - \beta \left(K'' \right)^{2} \lVert u \rVert_{H^{1}(M, g)}^{2} \\
& \geqslant \left(a - 2\beta \left(K'' \right)^{2} \right) \lVert \nabla_{g} u \rVert_{\calL^{2}(M, g)}^{2} +  \frac{\left(a - 2\beta \left(K'' \right)^{2} \right)}{a} \left( \int_{M} R_{g} u^{2} d\omega + \int_{\partial M} \frac{2a}{p-2} h_{g} u^{2} dS \right) \\
& \qquad + \frac{2\beta \left(K'' \right)^{2}}{a} \int_{M} R_{g} u^{2} d\omega + \frac{2\beta \left(K'' \right)^{2}}{a} \int_{\partial M} \frac{2a}{p-2} h_{g} u^{2} dS - 2\beta \left(K'' \right)^{2} \\
& \geqslant \frac{\left(a - \beta \left(K'' \right)^{2} \right)}{a} \left( a \lVert \nabla_{g} u \rVert_{\calL^{2}(M, g)}^{2} + \int_{M} R_{g} u^{2} d\omega + \int_{\partial M} \frac{2a}{p-2} h_{g} u^{2} dS  \right) \\
& \qquad - \frac{2\beta \left(K'' \right)^{2}}{a} \sup \lvert R_{g} \rvert \int_{M} u^{2} d\omega - 2\beta \left(K'' \right)^{2} \\
& = \frac{\left(a - 2\beta \left(K'' \right)^{2} \right)}{a} \eta_{1} - \frac{2\beta \left(K'' \right)^{2}}{a} \left( a + \sup \lvert R_{g} \rvert \right) = \eta_{1} - \frac{2\beta \left(K'' \right)^{2}}{a} \left( \eta_{1} + a + \sup \lvert R_{g} \rvert \right).
\end{align*}
Here we use (i) $ a - 2\beta \left(K'' \right)^{2} > 0 $ when $ \beta $ small enough; and (ii) $ h_{g} > 0 $ hence $ \frac{2\beta \left(K'' \right)^{2}}{a} \int_{\partial M} \frac{2}{p-2} h_{g} u^{2} dS > 0 $. We conclude from above derivation that
\begin{equation*}
\beta \ll 1 \Rightarrow I_{\beta}[u] > 0 \Rightarrow \eta_{1, \beta} = \inf_{u \in \mathcal{A}} I_{\beta}[u] > 0.
\end{equation*}
Take $ C_{\eta_{1}} $ to be the largest $ \beta $ such that above holds and it completes the proof.
\end{proof}
\medskip

\section{Monotone Iteration Method for Boundary Yamabe Problem}
With the tools equipped above, we show in this section the existence of a $ W^{2, q} $-solution of Yamabe equation (\ref{intro:eqn1}), provided the existence of sub-solution and super-solution of Yamabe equation when $ h_{g} > 0 $ everywhere on $ \partial M $.
\begin{theorem}\label{iteration:thm1}
Let $ (M, \partial M, g) $ be a compact manifold with non-empty smooth boundary. Let $ \nu $ be the unit outward normal vector along $ \partial M $ and $ q > \dim M $. Let $ h_{g} > 0 $ everywhere on $ \partial M $. Suppose that there exist $ u_{-} \in \calC_{0}(M) \cap H^{1}(M, g) $ and $ u_{+} \in W^{2, q}(M, g) \cap \calC_{0}(M) $, $ 0 \leqslant u_{-} \leqslant u_{+} $, $ u_{-} \not\equiv 0 $ on $ M $, some constant $ \lambda \neq 0 $ and some small enough positive constant $ \zeta > 0 $ such that
\begin{equation}\label{iteration:eqn1}
\begin{split}
-a\Delta_{g} u_{-} + R_{g} u_{-} - \lambda u_{-}^{p-1} & \leqslant 0 \; {\rm in} \; M, \frac{\partial u_{-}}{\partial \nu} + \frac{2}{p-2} h_{g} u_{-} = 0 \leqslant \frac{2}{p-2} \zeta u_{-}^{\frac{p}{2}} \; {\rm on} \; \partial M \\
-a\Delta_{g} u_{+} + R_{g} u_{+} - \lambda u_{+}^{p-1} & \geqslant 0 \; {\rm in} \; M, \frac{\partial u_{+}}{\partial \nu} + \frac{2}{p-2} h_{g} u_{+} \geqslant \frac{2}{p-2} \zeta u_{+}^{\frac{p}{2}} \; {\rm on} \; \partial M
\end{split}
\end{equation}
holds weakly. Then there exists a real, positive solution $ u \in \calC^{\infty}(M) \cap \calC^{1, \alpha}(M) $ of
\begin{equation}\label{iteration:eqn2}
\Box_{g} u = -a\Delta_{g} u + R_{g} u = \lambda u^{p-1}  \; {\rm in} \; M, B_{g} u =  \frac{\partial u}{\partial \nu} + \frac{2}{p-2} h_{g} u = \frac{2}{p-2} \zeta u^{\frac{p}{2}} \; {\rm on} \; \partial M.
\end{equation}
\end{theorem}
\begin{proof} Fix some $ q > \dim M $. Denote $ u_{0} = u_{+} $. Choose a constant $ A > 0 $ such that
\begin{equation}\label{iteration:eqn2a}
-R_{g}(x) + \lambda (p - 1) u(x)^{p-2} + A > 0, \forall u(x) \in [\min_{M} u_{-}(x), \max_{M} u_{+}(x) ], \forall x \in M
\end{equation}
pointwise. For the first step, consider the linear PDE
\begin{equation}\label{iteration:eqn3}
-a\Delta_{g} u_{1} + Au_{1} = Au_{0} - R_{g} u_{0} + \lambda u_{0}^{p-1}  \; {\rm in} \; M, \frac{\partial u_{1}}{\partial \nu} + \frac{2}{p-2} h_{g} u_{1} = \frac{2}{p-2} \zeta u_{0}^{\frac{p}{2}} \; {\rm on} \; \partial M.
\end{equation}
Since $ u_{0} = u_{+} \in W^{2, q}(M, g) \cap \calC_{0}(M) $, by Theorem \ref{pre:thm4}, there exists a unique solution $ u_{1} \in H^{1}(M, g) $. Since $ u_{0} \in W^{2, q}(M, g) \cap \calC_{0}(M) $, and thus $ u_{0} \in \calC^{1, \alpha}(M) \cap \calL^{q}(M, g) $ for all $ 1 < q < \infty $, it follows from $ \calL^{p}$-regularity in Theorem \ref{pre:thm1} that $ u_{1} \in W^{2, q}(M, g) $. By Sobolev embedding in Proposition \ref{pre:prop2}, it follows that $ u_{1} \in \calC^{1, \alpha}(M) $ for some $ \alpha \in (0, 1) $. We show that $ u_{1} \leqslant u_{0} = u_{+} $.
Subtracting the second equation in (\ref{iteration:eqn1}) and (\ref{iteration:eqn3}), we have
\begin{equation*}
\left( -a\Delta_{g} + A \right) (u_{0} - u_{1}) \geqslant 0 \; {\rm in} \; M, B_{g} (u_{0} - u_{1}) \geqslant 0 \; {\rm on} \; \partial M.
\end{equation*}
in the weak sense. Denote
\begin{equation*}
w = \max \lbrace 0, u_{1} - u_{0} \rbrace.
\end{equation*}
It is immediate that $ w \in H^{1}(M, g) \cap \calC_{0}(M) $ and $ w \geqslant 0 $. It follows that
\begin{align*}
0 & \geqslant \int_{M} \left( a \nabla_{g} (u_{1} - u_{0}) \cdot \nabla_{g} w + A(u_{1} - u_{0}) w \right) d\omega + \int_{\partial M} \frac{2}{p-2} h_{g} (u_{1} - u_{0}) w dS \\
& = \int_{M} \left( a \lvert \nabla_{g} w \rvert^{2} + A w^{2} \right) d\omega + \int_{\partial M} \frac{2}{p-2} h_{g} w^{2} dS \geqslant 0.
\end{align*}
The last inequality holds since $ A > 0 $ and $ h_{g} > 0 $ everywhere on $ \partial M $. It follows that
\begin{equation*}
w \equiv 0 \Rightarrow 0 \geqslant u_{1} - u_{0} \Rightarrow u_{0} \geqslant u_{1}.
\end{equation*}
By the same argument, we can show that $ u_{1} \geqslant u_{-} $ and hence $ u_{-} \leqslant u_{1} \leqslant u_{+} $. Assume inductively that $ u_{-} \leqslant \dotso \leqslant u_{k-1} \leqslant u_{k} \leqslant u_{+} $ for some $ k > 1 $ with $ u_{k} \in W^{2, q}(M, g) $ , the $ (k + 1)th $ iteration step is
\begin{equation}\label{iteration:eqn4}
-a\Delta_{g} u_{k+1} + Au_{k+1} = Au_{k} - R_{g} u_{k} + \lambda u_{k}^{p-1}  \; {\rm in} \; M, \frac{\partial u_{k+1}}{\partial \nu} + \frac{2}{p-2} h_{g} u_{k+1} = \frac{2}{p-2} \zeta u_{k}^{\frac{p}{2}} \; {\rm on} \; \partial M.
\end{equation}
Since $ u_{k} \in W^{2, q}(M, g) $ thus $ u_{k} \in \calC^{1, \alpha}(M) $ due to Sobolev embedding in Proposition \ref{pre:prop2}, by Theorem \ref{pre:thm1} and \ref{pre:thm4} again, we conclude that there exists $ u_{k+1} \in W^{2, q}(M, g) $ that solves (\ref{iteration:eqn4}). In particular, $ u_{k}^{\frac{p}{2}} = u_{k}^{\frac{n}{n-2}} \in \calC^{1}(M) $ hence the hypothesis of the boundary condition in Theorem \ref{pre:thm1} and \ref{pre:thm4} are satisfied. We show that $ u_{-} \leqslant u_{k + 1} \leqslant u_{k} \leqslant u_{+} $. The $ kth $ iteration step 
\begin{equation}\label{iteration:eqn5}
-a\Delta_{g} u_{k} + Au_{k} = Au_{k - 1} - R_{g} u_{k - 1} + \lambda u_{k - 1}^{p-1}  \; {\rm in} \; M, \frac{\partial u_{k}}{\partial \nu} + \frac{2}{p-2} h_{g} u_{k} = \frac{2}{p-2} \zeta u_{k - 1}^{\frac{p}{2}} \; {\rm on} \; \partial M.
\end{equation}
Subtracting (\ref{iteration:eqn4}) by (\ref{iteration:eqn5}), we conclude that
\begin{align*}
& \left( -a\Delta_{g} + A \right) \left( u_{k + 1} - u_{k} \right) = A(u_{k} - u_{k - 1}) - R_{g} (u_{k} - u_{k - 1}) + \lambda \left( u_{k}^{p-1} - u_{k - 1}^{p-1} \right) \leqslant 0 \; {\rm in} \; M; \\
& \frac{\partial \left(u_{k+1} - u_{k}\right)}{\partial \nu} + \frac{2}{p-2} h_{g} (u_{k+ 1} - u_{k}) = \frac{2}{p-2} \zeta u_{k}^{\frac{p}{2}}  - \frac{2}{p-2} \zeta u_{k - 1}^{\frac{p}{2}} \leqslant 0 \; {\rm on} \; \partial M.
\end{align*}
By induction we have $ u_{-} \leqslant u_{k} \leqslant u_{k - 1} \leqslant u_{+} $. The first inequality above is then due to the pointwise mean value theorem and the choice of $ A $ in (\ref{iteration:eqn2a}). The second inequality is immediate. Note that since both $ u_{k}, u_{k-1} \in W^{2, q}(M, g) $, above inequalities hold in strong sense. We choose
\begin{equation*}
\tilde{w} = \max \lbrace 0, u_{k+1} - u_{k} \rbrace.
\end{equation*}
Clearly $ \tilde{w} \geqslant 0 $ with $ \tilde{w} \in H^{1}(M, g) \cap \calC_{0}(M) $. Pairing $ \tilde{w} $ with $ \left( -a\Delta_{g} + A \right) \left( u_{k} - u_{k + 1} \right) \leqslant 0 $, we have
\begin{align*}
0 & \geqslant \int_{M} \left( -a\Delta_{g} + A \right) \left( u_{k+1} - u_{k} \right)  \tilde{w} d\omega = \int_{M} a \nabla_{g} (u_{k+1} - u_{k}) \cdot \nabla_{g} \tilde{w} d\omega - \int_{\partial M} \frac{\partial \left(u_{k+1} - u_{k}\right)}{\partial \nu} \tilde{w} dS \\
& \geqslant  \int_{M} a \nabla_{g} (u_{k+1} - u_{k}) \cdot \nabla_{g} \tilde{w} d\omega + \int_{\partial M} \frac{2a}{p -2} h_{g} (u_{k+1} - u_{k}) \tilde{w} dS \\
& = a \lVert \nabla_{g} \tilde{w} \rVert_{\calL^{2}(M, g)}^{2} + \frac{2a}{p-2} \int_{\partial M} h_{g} \tilde{w}^{2} dS \geqslant 0.
\end{align*}
It follows that
\begin{equation*}
\tilde{w} = 0 \Rightarrow 0 \geqslant u_{k + 1} - u_{k} \Rightarrow u_{k + 1} \leqslant u_{k}.
\end{equation*}
By the same argument and the induction $ u_{k} \geqslant u_{-} $, we conclude that $ u_{k+1} \geqslant u_{-} $. Thus
\begin{equation}\label{iteration:eqn6}
0 \leqslant u_{-} \leqslant u_{k+1} \leqslant u_{k} \leqslant u_{+}, u_{k} \in W^{2, q}(M, g), \forall k \in \mathbb{N}.
\end{equation}
By Theorem \ref{pre:thm4}, the operator $ -a\Delta_{g} + A $ is injective. Applying $ L^{p} $-regularity in Theorem \ref{pre:thm2}, we conclude from the first iteration step (\ref{iteration:eqn3}) that
\begin{equation}\label{iteration:eqn7}
\lVert u_{1} \rVert_{W^{2, q}(M, g)} \leqslant C' \left( \left\lVert Au_{0} - R_{g} u_{0} + \lambda u_{0}^{p-1} \right\rVert_{\calL^{q}(M, g)} + \left\lVert \frac{2}{p-2} \zeta u_{0}^{\frac{p}{2}} \right\rVert_{W^{1, q}(M, g)} \right).
\end{equation}
Choose $ \zeta > 0 $ small enough so that
\begin{equation}\label{iteration:eqn8}
\begin{split}
& \zeta \cdot \frac{2}{p-2} \cdot \left\lVert u_{0}^{\frac{p}{2}} \right\rVert_{W^{1, q}(M, g)} \leqslant 1; \\
& \zeta \cdot \frac{2}{p-2} \sup_{M} \left( u_{0}^{\frac{p}{2}} \right) \cdot Vol_{g}(M) \\
& \qquad +\zeta \cdot \frac{2}{p-2} \cdot \frac{p}{2} \sup_{M} \left( u_{0}^{\frac{p - 2}{2}} \right) C' \left( \left( A  + \sup_{M} \lvert R_{g} \rvert + \lvert \lambda \rvert \sup_{M} \left( u_{0}^{p-2} \right) \right) \sup_{M} \left( u_{0} \right)\cdot \text{Vol}_{g}(M) + 1 \right) \leqslant 1.
\end{split}
 \end{equation}
Note that for smaller $ \zeta $, the sub-solution and super-solution in (\ref{iteration:eqn1}) still hold. Due to (\ref{iteration:eqn8}), we conclude that
\begin{align*}
\lVert u_{1} \rVert_{W^{2, q}(M, g)} & \leqslant C' \left( \left\lVert Au_{0} - R_{g} u_{0} + \lambda u_{0}^{p-1} \right\rVert_{\calL^{q}(M, g)} + 1 \right) \\
& \leqslant C' \left( \left( A  + \sup_{M} \lvert R_{g} \rvert + \lvert \lambda \rvert \sup_{M} \left( u_{0}^{p-2} \right) \right) \sup_{M} \left( u_{0} \right) \cdot \text{Vol}_{g}(M) + 1 \right).
\end{align*}
Inductively, we assume
\begin{equation}\label{iteration:eqn9}
\lVert u_{k} \rVert_{W^{2, q}(M, g)} \leqslant C' \left( \left( A  + \sup_{M} \lvert R_{g} \rvert + \lvert \lambda \rvert \sup_{M} \left( u_{0}^{p-2} \right) \right) \sup_{M} \left( u_{0} \right) \cdot \text{Vol}_{g}(M) + 1 \right).
\end{equation}
For $ u_{k + 1} $, we conclude from (\ref{iteration:eqn4}) that
\begin{equation}\label{iteration:eqn10}
\lVert u_{k+1} \rVert_{W^{2, q}(M, g)} \leqslant C' \left( \left\lVert Au_{k} - R_{g} u_{k} + \lambda u_{k}^{p-1} \right\rVert_{\calL^{q}(M, g)} + \left\lVert \frac{2}{p-2} \zeta u_{k}^{\frac{p}{2}} \right\rVert_{W^{1, q}(M, g)} \right).
\end{equation}
The last term in (\ref{iteration:eqn10}) can be estimated as
\begin{align*}
\left\lVert \frac{2}{p-2} \zeta u_{k}^{\frac{p}{2}} \right\rVert_{W^{1, q}(M, g)} & = \frac{2}{p-2} \zeta \left( \left\lVert u_{k}^{\frac{p}{2}} \right\rVert_{\calL^{q}(M, g)} + \left\lVert \nabla_{g} \left( u_{k}^{\frac{p}{2}} \right) \right\rVert_{\calL^{q}(M, g)} \right) \\
& \leqslant  \frac{2}{p-2} \zeta \left( \sup_{M} \left( u_{k}^{\frac{p}{2}} \right) \cdot Vol_{g}(M) +\frac{p}{2} \sup_{M} \left( u_{k}^{\frac{p - 2}{2}} \right) \left\lVert \nabla_{g} u_{k} \right\rVert_{\calL^{q}(M, g)} \right) \\
& \leqslant  \frac{2}{p-2} \zeta \left( \sup_{M} \left( u_{0}^{\frac{p}{2}} \right) \cdot Vol_{g}(M) +\frac{p}{2} \sup_{M} \left( u_{0}^{\frac{p - 2}{2}} \right) \lVert u_{k}  \rVert_{W^{2, q}(M, g)} \right).
\end{align*}
By the choice of $ \zeta $ in (\ref{iteration:eqn8}) and induction assumption in (\ref{iteration:eqn9}), we conclude that
\begin{equation*}
\left\lVert \frac{2}{p-2} \zeta u_{k}^{\frac{p}{2}} \right\rVert_{W^{1, q}(M, g)} \leqslant 1.
\end{equation*}
It follows from (\ref{iteration:eqn10}) that
\begin{equation}\label{iteration:eqn11}
\begin{split}
\lVert u_{k+1} \rVert_{W^{2, q}(M, g)} & \leqslant C' \left( \left\lVert Au_{k} - R_{g} u_{k} + \lambda u_{k}^{p-1} \right\rVert_{\calL^{q}(M, g)} + \left\lVert \frac{2}{p-2} \zeta u_{k}^{\frac{p}{2}} \right\rVert_{W^{1, q}(M, g)} \right) \\
& \leqslant C' \left( \left( A  + \sup_{M} \lvert R_{g} \rvert + \lvert \lambda \rvert \sup_{M} \left( u_{0}^{p-2} \right) \right) \sup_{M} \left( u_{0} \right) \cdot \text{Vol}_{g}(M) + 1 \right).
\end{split}
\end{equation}
It follows that the sequence $ \lbrace u_{k} \rbrace_{k \in \mathbb{N}} $ is uniformly bounded in $ W^{2, q} $-norm. By Sobolev embedding in Proposition \ref{pre:prop2}, we conclude that the same sequence is uniformly bounded in $ \calC^{1} $-norm. Thus by Arzela-Ascoli theorem, we conclude that there exists $ u $ such that
\begin{equation*}
u = \lim_{k \rightarrow \infty} u_{k}, 0 \leqslant u_{-} \leqslant u \leqslant u_{+}, \Box_{g} u = \lambda u^{p-1}  \; {\rm in} \; M, B_{g} u = \frac{2}{p-2} \zeta u^{\frac{p}{2}} \; {\rm on} \; \partial M.
\end{equation*}
Apply elliptic regularity, we conclude that $ u \in W^{2, q}(M, g) $. A standard bootstrapping argument concludes that $ u \in \calC^{\infty}(M) \cap \calC^{1, \alpha}(M) $, due to Schauder estimates. The regularity of $ u $ on $ \partial M $ is determined by $ u^{p-1} $.

Lastly we show that $ u $ is positive. Since $ u \in \calC^{\infty}(M) $ it is smooth locally, the local strong maximum principle says that if $ u = 0 $ in some interior domain $ \Omega $ then $ u \equiv 0 $ on $ \Omega $, a continuation argument then shows that $ u \equiv 0 $ in $ M $. But $ u \geqslant u_{-} $ and $ u_{-} > 0 $ within some region. Thus $ u > 0 $ in the interior $ M $. By the same argument in \cite[\S1]{ESC}, we conclude that $ u > 0 $ on $ M $.
\end{proof}
\begin{remark}\label{iteration:re1}
The choice of $ \zeta $ in the proof of Theorem \ref{iteration:thm1} is uniform for all $ k $. We may need to choice $ \zeta $ small enough a priori to construct the super-solution in (\ref{iteration:eqn1}). Therefore we need to shrink $ \zeta $ totally two times. It is worth mentioning that the choice of sub-solution $ u_{-} $ in Theorem \ref{iteration:thm1} is very special, requiring a homogeneous Robin condition on $ \partial M $, hence the boundary condition in (\ref{iteration:eqn1}) holds for arbitrarily small $ \zeta > 0 $.
\end{remark}
\medskip

As shown in Theorem 4.3 of \cite{XU3} and Theorem 5.5 of \cite{XU4}, we need a perturbed boundary Yamabe equation
\begin{equation*}
-a\Delta_{g} u + \left(R_{g} + \tau \right) u = \lambda u^{p-1} \; {\rm in} \; M, \frac{\partial u}{\partial \nu} + \frac{2}{p-2} h_{g} u = \frac{2}{p-2} \zeta u^{\frac{p}{2}} \; {\rm on} \; \partial M
\end{equation*}
when $ \eta_{1} > 0 $. The existence of the positive solution can be shown by the same monotone iteration scheme as above, provided the existence of the sub- and super-solutions.
\begin{corollary}\label{iteration:cor1}
Let $ (M, \partial M, g) $ be a compact manifold with non-empty smooth boundary. Let $ \nu $ be the unit outward normal vector along $ \partial M $ and $ q > \dim M $. Let $ h_{g} > 0 $ everywhere on $ \partial M $. Let $ \tau < 0 $ be a negative constant. Suppose that there exist $ u_{-} \in \calC_{0}(M) \cap H^{1}(M, g) $ and $ u_{+} \in W^{2, q}(M, g) \cap \calC_{0}(M) $, $ 0 \leqslant u_{-} \leqslant u_{+} $, $ u_{-} \not\equiv 0 $ on $ M $, some constant $ \lambda \neq 0 $ and some small enough positive constant $ \zeta > 0 $ such that
\begin{equation}\label{iteration:eqn12}
\begin{split}
-a\Delta_{g} u_{-} + \left(R_{g} + \tau \right) u_{-} - \lambda u_{-}^{p-1} & \leqslant 0 \; {\rm in} \; M, \frac{\partial u_{-}}{\partial \nu} + \frac{2}{p-2} h_{g} u_{-} = 0 \leqslant \frac{2}{p-2} \zeta u_{-}^{\frac{p}{2}} \; {\rm on} \; \partial M \\
-a\Delta_{g} u_{+} + \left( R_{g} + \tau \right) u_{+} - \lambda u_{+}^{p-1} & \geqslant 0 \; {\rm in} \; M, \frac{\partial u_{+}}{\partial \nu} + \frac{2}{p-2} h_{g} u_{+} \geqslant \frac{2}{p-2} \zeta u_{+}^{\frac{p}{2}} \; {\rm on} \; \partial M
\end{split}
\end{equation}
holds weakly. Then there exists a real, positive solution $ u \in \calC^{\infty}(M) \cap \calC^{1, \alpha}(M) $ of
\begin{equation}\label{iteration:eqn13}
\Box_{g, \tau} u : = -a\Delta_{g} u + \left( R_{g} + \tau \right) u = \lambda u^{p-1}  \; {\rm in} \; M, B_{g} u =  \frac{\partial u}{\partial \nu} + \frac{2}{p-2} h_{g} u = \frac{2}{p-2} \zeta u^{\frac{p}{2}} \; {\rm on} \; \partial M.
\end{equation}
\end{corollary}
\begin{proof}
Everything is exactly the same as in Theorem \ref{iteration:thm1} except replacing $ R_{g} $ by $ R_{g} + \tau $, and choosing a slightly different constant $ A > 0 $ such that
\begin{equation*}
-R_{g}(x) - \tau + \lambda (p - 1) u(x)^{p-2} + A > 0, \forall u(x) \in [\min_{M} u_{-}(x), \max_{M} u_{+}(x) ], \forall x \in M.
\end{equation*}
\end{proof}
\medskip

\section{Scalar and Mean Curvatures under Conformal Change}
The signs of scalar and mean curvatures are critical in barrier methods, not only in the solvability in Proposition \ref{pre:prop4}, but also in the monotone iteration scheme, Theorem \ref{iteration:thm1}. Results in this sections shows how to convert general $ R_{g} $ and $ h_{g} $ to special ones with desired signs under conformal change. All results are proved in \cite[\S4]{XU3} and \cite[\S5]{XU4}, hence we omit the proofs here. Throughout this section, we always assume $ (M, \partial M, g) $ to be a compact manifold with smooth boundary $ \partial M $, $ n =  \dim M \geqslant 3 $.
\medskip

The first two result below converts general $ h_{g} $ on $ \partial M $ to a mean curvature which is always positive or negative under conformal change.
\begin{theorem}\label{conformal:thm1}
Let $ (M, \partial M, g) $ be a compact manifold with non-empty smooth boundary. There exists a conformal metric $ \tilde{g} $ with mean curvature $ \tilde{h} > 0 $ everywhere on $ \partial M $.
\end{theorem}
\begin{corollary}\label{conformal:cor1}
Let $ (M, \partial M, g) $ be a compact manifold with non-empty smooth boundary. There exists a conformal metric $ \tilde{g} $ with mean curvature $ \tilde{h} < 0 $ everywhere on $ \partial M $.
\end{corollary}
\medskip

Next two results concerns the signs of $ R_{g} $ and $ h_{g} $ simultaneously, which converts $ R_{g} $, either nonnegative or nonpositive everywhere, and a general $ h_{g} $ to a scalar curvature negative somewhere meanwhile keeping the sign of $ h_{g} $ unchanged pointwise, under conformal change.
\begin{theorem}\label{conformal:thm2}
Let $ (M, \partial M, g) $ be a compact manifold with smooth boundary. Let $ R_{g} \geqslant 0 $ everywhere. There exists a conformal metric $ \tilde{g} $ with scalar curvature $ \tilde{R} $ and mean curvature $ \tilde{h} $ such that $ \tilde{R} < 0 $ somewhere, and $ \text{sgn}(h_{g}) = \text{sgn}(\tilde{h}) $ pointwise on $ \partial M $.
\end{theorem}
\begin{corollary}\label{conformal:cor2}
Let $ (M, \partial M, g) $ be a compact manifold with non-empty smooth boundary. Let $ R_{g} \leqslant 0 $ everywhere. There exists a conformal metric $ \tilde{g} $ with scalar curvature $ \tilde{R} $ and mean curvature $ \tilde{h} $ such that $ \tilde{R} > 0 $ somewhere, and $ \text{sgn}(h_{g}) = \text{sgn}(\tilde{h}) $ pointwise on $ \partial M $.
\end{corollary}
\medskip

\section{Boundary Yamabe Problem with Minimal Boundary Case}
Recall the boundary Yamabe problem for general case
\begin{equation}\label{yamabe:eqn1}
\begin{split}
\Box_{g} u & : =  -a\Delta_{g} u + R_{g} u = \lambda u^{p-1} \; {\rm in} \; M; \\
B_{g} u & : =  \frac{\partial u}{\partial \nu} + \frac{2}{p-2} h_{g} u = \frac{2}{p-2} \zeta u^{\frac{p}{2}} \; {\rm on} \; \partial M.
\end{split}
\end{equation}
Note that $ \lambda $ and $ \zeta $ are the constant scalar and mean curvature, respectively, with respect to $ \tilde{g} = u^{p-2} g $ for some real, smooth function $ u > 0 $. In this section, we apply the sub-solution and super-solution technique in Theorem \ref{iteration:thm1} to solve boundary Yamabe equation for five cases:
\begin{enumerate}[(A).]
\item $ \eta_{1} = 0 $;
\item $ \eta_{1} < 0 $ with $ h_{g} > 0 $ everywhere on $ \partial M $ and arbitrary $ R_{g} $;
\item $ \eta_{1} < 0 $ with arbitrary $ h_{g} $ and $ R_{g} $;
\item $ \eta_{1} > 0 $ with $ h_{g} > 0 $ everywhere and $ R_{g} < 0 $ somewhere;
\item $ \eta_{1} > 0 $ with arbitrary $ h_{g} $ and $ R_{g} $.
\end{enumerate}
Throughout this section, we assume $ \dim M \geqslant 3 $. We always assume that $ (M, \partial M, g) $ be a compact manifold with smooth boundary, and $ \nu $ be the unit outward normal vector along $ \partial M $. Note that case (C) can be converted to case (B) under a one-step conformal change by Theorem \ref{conformal:thm1}, and case (E) can be converted to case (D) under at most two-step conformal change by Theorem \ref{conformal:thm2}. Details for $ (E) \rightarrow (D) $ and $ (C) \rightarrow (B) $ can be found in \cite[\S5]{XU4}. Hence we only prove cases (A), (B), and (D) in this section. Similar to positive eigenvalue cases in \cite{XU4}, \cite{XU3}, we need a result of perturbed boundary Yamabe equation
\begin{equation}\label{yamabe:eqns1}
-a\Delta_{g} u_{\tau} + \left(R_{g} + \tau \right) u_{\tau} = \lambda_{\tau} u_{\tau}^{p - 1} \; {\rm in} \; M, \frac{\partial u_{\tau}}{\partial \nu} + \frac{2}{p - 2} h_{g} u_{\tau} = \frac{2}{p-2} \zeta u_{\tau}^{\frac{p}{2}} \; {\rm on} \; \partial M
\end{equation}
with $ \tau < 0 $, $ \zeta > 0 $, $ h_{g} > 0 $ everywhere on $ \partial M $, and $ \lambda_{\tau} > 0 $ defined as
\begin{equation}\label{yamabe:eqns2}
\lambda_{\tau} = \inf_{u \neq 0} \frac{\int_{M} a \lvert \nabla_{g} u \rvert^{2} d\omega + \int_{M} \left( R_{g} + \tau \right) u^{2} d\omega + \int_{\partial M} \frac{2a}{p-2} h_{g} u^{2} dS}{\left( \int_{M} u^{p} d\omega \right)^{\frac{2}{p}}}.
\end{equation}
Recall that
\begin{equation*}
\lambda(M) = \inf_{u \neq 0} \frac{\int_{M} a \lvert \nabla_{g} u \rvert^{2} d\omega + \int_{M} R_{g} u^{2} d\omega + \int_{\partial M} \frac{2a}{p-2} h_{g} u^{2} dS}{\left( \int_{M} u^{p} d\omega \right)^{\frac{2}{p}}}
\end{equation*}
It is immediate that $ \eta_{1} > 0 $ implies $ \lambda(M) > 0 $. By the same manner as in the proofs of Theorem \ref{eigen:thm1} and Lemma 4.1 of \cite{XU3}, we conclude that $ \lambda_{\tau} > 0 $ when $ \lvert \tau \rvert $ is small enough with $ \tau < 0 $.
\medskip

Case (A) is just a special eigenvalue problem when $ \eta_{1} = 0 $.
\begin{theorem}\label{yamabe:thm1}
Let $ (M, \partial M, g) $ be a compact manifold with non-empty smooth boundary and $ \eta_{1} = 0 $. Then the boundary Yamabe equation (\ref{yamabe:eqn1}) has a real, positive, smooth solution with $ \lambda = \zeta = 0 $.
\end{theorem}
\begin{proof} It is an immediate consequence of the eigenvalue problem (\ref{eigen:eqn1}) with $ \eta_{1} = 0 $.
\end{proof}
\medskip

Next theorem is related to the existence of solution of (\ref{yamabe:eqn1}) with $ \eta_{1} < 0 $ and $ h_{g} > 0 $ everywhere on $ \partial M $.
\begin{theorem}\label{yamabe:thm2}
Let $ (M, \partial M, g) $ be a compact manifold with non-empty smooth boundary. Let $ h_{g} > 0 $ everywhere on $ \partial M $. When $ \eta_{1} < 0 $, there exists some $ \lambda < 0 $ and $ \zeta > 0 $ such that the boundary Yamabe equation (\ref{yamabe:eqn1}) has a real, positive solution $ u \in \calC^{\infty}(M) $.
\end{theorem}
\begin{proof} The first step is to determine the choice of $ \lambda $ and construct the sub-solution. By (\ref{eigen:eqn1}), there exists a real, positive function $ \phi \in \calC^{\infty}(M) $ satisfying
\begin{equation}\label{yamabe:eqn2}
-a\Delta_{g} \phi + R_{g} \phi = \eta_{1} \phi \; {\rm in} \; M, \frac{\partial \phi}{\partial \nu} + \frac{2}{p-2} h_{g} \phi = 0 \; {\rm on} \; \partial M
\end{equation}
with $ \eta_{1} < 0 $. Scaling $ \phi \mapsto t\phi $, $ t < 1 $, we may assume that $ \sup_{M} \phi < 1 $. It follows that $ \phi^{p-1} < \phi $. Hence $ \eta_{1} \phi < \eta_{1} \phi^{p-1} $. Choose the negative constant $ \lambda \in (\eta_{1}, 0) $, we conclude from (\ref{yamabe:eqn2}) that
\begin{equation*}
-a\Delta_{g} \phi + R_{g} \phi = \eta_{1} \phi  \leqslant \eta_{1} \phi^{p-1} \leqslant \lambda \phi^{p-1}.
\end{equation*}
Define
\begin{equation}\label{yamabe:eqn3}
u_{-} : = \phi \; {\rm on} \; M.
\end{equation}
We conclude that
\begin{equation}\label{yamabe:eqn4}
-a\Delta_{g} u_{-} + R_{g} u_{-} \leqslant \lambda u_{-}^{p-1} \; {\rm in} \; M, \frac{\partial u_{-}}{\partial \nu} + \frac{2}{p-2} h_{g} u_{-} = 0 \leqslant \frac{2}{p-2} \zeta u_{-}^{\frac{p}{2}} \; {\rm on} \; \partial M.
\end{equation}
Note that this holds for any choice of $ \zeta > 0 $. It is immediate to see that $ u_{-} > 0 $ is a real, smooth function on $ M $.
\medskip

Next we determine the upper bound of $ \zeta $ and construct the super-solution. An upper bound of $ \zeta $ is enough to apply monotone iteration scheme in Theorem \ref{iteration:thm1}, as discussed in that theorem. Select
\begin{equation}\label{yamabe:eqn5}
K_{1}^{p-2} = \max \left\lbrace \frac{\inf_{M} R_{g}}{\lambda}, \sup_{M} u_{-}^{p-2} \right\rbrace.
\end{equation}
Note that the quantity $ \frac{\inf_{M} R_{g}}{\lambda} $ is negative if $ R_{g} \geqslant 0 $ everywhere. Then select $ \zeta $ such that
\begin{equation}\label{yamabe:eqn6}
\zeta \leqslant \left( \inf_{\partial M} h_{g} \right) \cdot K_{1}^{\frac{2-p}{2}}.
\end{equation}
We define
\begin{equation}\label{yamabe:eqn7}
u_{+} : = K_{1}.
\end{equation}
Clearly $ u_{+} > 0 $ is a real, smooth function on $ M $. By (\ref{yamabe:eqn5}), we conclude that $ u_{+} \geqslant u_{-} > 0 $. We show that $ u_{+} $ is a super-solution of (\ref{yamabe:eqn1}). By (\ref{yamabe:eqn5}) again, we observe that
\begin{equation*}
-a\Delta_{g} u_{+} + R_{g} u_{+} = R_{g} K_{1} \geqslant \left( \inf_{M} R_{g} \right) K_{1} \geqslant \lambda K_{1}^{p-1} = \lambda u_{+}^{p-1}.
\end{equation*}
By (\ref{yamabe:eqn6}), we see that
\begin{equation*}
\frac{\partial u_{+}}{\partial \nu} + \frac{2}{p-2} h_{g} u_{+} = \frac{2}{p-2} h_{g} K_{1} \geqslant \frac{2}{p-2} \left( \inf_{\partial M} h_{g} \right)  K_{1} \geqslant \frac{2}{p-2} \zeta K_{1}^{\frac{p}{2}} = \frac{2}{p-2} \zeta u_{+}^{\frac{p}{2}}.
\end{equation*}
By Theorem \ref{iteration:thm1}, we conclude that there exists a real, positive function $ u \in \calC^{\infty}(M) \cap \calC^{1, \alpha}(M) $ that solves (\ref{yamabe:eqn1}), with a further choice of $ \zeta $. The choice of $ \alpha $ depends on the dimension $ n $ and the nonlinear power $ p - 1 $.
\end{proof}
\medskip

Lastly we show the existence of solution of (\ref{yamabe:eqn1}) when $ \eta_{1} > 0 $, $ h_{g} > 0 $ everywhere on $ \partial M $ and $ R_{g} < 0 $ somewhere in $ M $. As mentioned above, we first show the existence of the solution of (\ref{yamabe:eqns1}) by monotone iteration scheme in Corollary \ref{iteration:cor1}. Existence of solution of a local Dirichlet problem of perturbed Yamabe equation in Proposition \ref{pre:prop4} plays a central role in the following theorem.
\begin{remark}\label{yamabe:sre}
The result below relies on the construction of a smooth super-solution of (\ref{yamabe:eqns1}), which is based on a delicate gluing skill. Roughly speaking, we construct an ``almost" super-solution initially. It is only an ``almost" super-solution since the sub-solution has a spike within a small region. We then glue the ``almost" super-solution with the spike, which requires a small perturbation within a small region. A delicate choice of perturbation gives us good enough control of the Laplacian, and the zeroth order terms. The Laplacian is controllable since it reflects the deviation of the evaluation of a function at a point from its average within a small ball centered at the same point. Therefore a small perturbation of the function will not change the average dramatically, intuitively. Fortunately, our PDE only involves in the Laplacian and the zeroth order terms.
\end{remark}
\begin{theorem}\label{yamabe:thm3}
Let $ (M, \partial M, g) $ be a compact manifold with non-empty smooth boundary. Let $ \tau < 0 $ be a negative constant. Assume $ R_{g} < 0 $ somewhere in $ M $ and $ h_{g} > 0 $ everywhere on $ \partial M $. When $ \eta_{1} > 0 $ and $ \lvert \tau \rvert $ is small enough, there exists some $ \lambda > 0 $ and $ \zeta > 0 $ such that the perturbed boundary Yamabe equation (\ref{yamabe:eqns1}) has a real, positive solution $ u \in \calC^{\infty}(M) $.
\end{theorem}
\begin{proof}
We apply Corollary \ref{iteration:cor1} to show this by constructing sub- and super-solutions. Due to the discussion above, $ \eta_{1} > 0 $ implies $ \lambda_{\tau} > 0 $ when $ \lvert \tau \rvert $ is small enough. When $ \lvert \beta \rvert $ is small enough, it follows from Theorem \ref{eigen:thm1} that the following eigenvalue problem with $ \eta_{1, \beta} > 0 $
\begin{equation}\label{yamabe:eqns3}
-a\Delta_{g} \varphi + \left(R_{g} + \tau \right) \varphi = \eta_{1, \beta} \varphi + \tau \varphi \; {\rm in} \; M, \frac{\partial \varphi}{\partial \nu} + \left(\frac{2}{p-2}h_{g} - \beta \right) \varphi = 0 \; {\rm on} \; \partial M
\end{equation}
admits a positive solution $ \varphi \in \calC^{\infty}(M) $. Fix this $ \beta $ from now on. Let $ \tau $ even smaller if necessary, we have $ \eta_{1, \beta} + \tau > 0 $. Thus $ \eta_{1, \beta} + \tau > 0, \forall \tau \in [\tau_{0}, 0] $ if it holds for some $ \tau_{0} $. Observe that any scaling $ \delta \varphi $ is also an eigenfunction with respect to $ \eta_{1, \beta} $. For the given $ \lambda_{\tau} $, we want
\begin{equation*}
\left( \eta_{1, \beta} + \tau \right) \inf_{M} (\delta \varphi) > 2^{p-2} \lambda_{\tau} \sup_{M} \left(\delta^{p-1} \varphi^{p-1} \right) \Leftrightarrow \frac{\left( \eta_{1, \beta} + \tau \right)}{2^{p-2}\lambda_{\tau}} > \delta^{p-2} \frac{\sup_{M} \varphi^{p-1}}{\inf_{M} \varphi}.
\end{equation*}
For fixed $ \eta_{1, \beta}, \lambda_{\tau}, \varphi, \beta $, this can be done by letting $ \delta $ small enough. We denote $ \phi = \delta \varphi $. It follows that
\begin{equation}\label{yamabe:eqns4}
\begin{split}
-a\Delta_{g} \phi + \left( R_{g} + \tau \right) \phi & = \left( \eta_{1, \beta} + \tau \right) \phi \; {\rm in} \; M; \\
\left( \eta_{1, \beta} + \tau \right) \inf_{M} \phi & > 2^{p-2} \lambda_{\tau} \sup_{M} \phi^{p-1} \geqslant 2^{p-2} \lambda_{\tau} \phi^{p-1} > \lambda_{\tau} \phi^{p-1} \; {\rm in} \; M.
\end{split}
\end{equation}
Set
\begin{equation}\label{yamabe:eqns5}
\theta = \left( \eta_{1, \beta} + \tau \right) \sup_{M} \phi - 2^{p-2} \lambda_{\tau} \inf_{M} \phi^{p-1} > \left( \eta_{1, \beta} + \tau \right)\phi - 2^{p-2} \lambda_{\tau} \phi^{p-1} \; \text{pointwise}.
\end{equation}
Thus we have
\begin{equation}\label{yamabe:eqns6}
\begin{split}
-a\Delta_{g} \phi + \left(R_{g} + \tau \right) \phi & = \left( \eta_{1, \beta} + \tau \right) \phi > 2^{p-2} \lambda_{\tau} \phi^{p-1} > \lambda_{\tau} \phi^{p-1} \; {\rm in} \; M \; {\rm pointwise}; \\
\frac{\partial \phi}{\partial \nu} + \left( \frac{2}{p-2}h_{g} - \beta \right) \phi & = 0 \; {\rm on} \; \partial M.
\end{split}
\end{equation}
\medskip

We now construct the sub-solution with respect to $ \lambda = \lambda_{\tau} $ in Proposition \ref{pre:prop4}. Pick up a small enough interior Riemannian domain $ (\Omega, g) $ in which $ R_{g} < 0 $ such that the Dirichlet boundary value problem (\ref{pre:eqn17}) with the given $ \lambda_{\tau} $ above has a positive solution $ u_{1} \in \calC_{0}(\Omega) \cap H_{0}^{1}(\Omega, g) $, i.e.
\begin{equation}\label{yamabe:eqn8}
-a\Delta_{g} u_{1} + \left(R_{g} + \tau \right) u_{1} = \lambda_{\tau} u_{1}^{p-1} \; {\rm in} \; \Omega, u_{1} = 0 \; {\rm on} \; \partial M.
\end{equation}
Extend $ u_{1} $ by zero on the rest of $ M $, we define
\begin{equation}\label{yamabe:eqn9}
u_{-} : = \begin{cases} u_{1}, & \text{within} \; \Omega; \\ 0, & \text{outside} \; \Omega. \end{cases}
\end{equation}
By the same argument in Theorem 5.5 of \cite{XU4}, we conclude that $ u_{-} \in \calC_{0}(M) \cap H^{1}(M, g) $; in addition, $ u_{-} $ is a sub-solution of (\ref{yamabe:eqns1}) due to the fact that $ \Omega $ is an interior subset of $ M $. We conclude that
\begin{equation}\label{yamabe:eqn10}
-a\Delta_{g} u_{-} + \left(R_{g} + \tau \right) u_{-} \leqslant \lambda_{\tau} u_{-}^{p-1} \; {\rm in} \; M, \frac{\partial u_{-}}{\partial \nu} + \frac{2}{p-2} h_{g} u_{-} = 0 \leqslant \frac{2}{p-2} \zeta u_{-}^{\frac{p}{2}} \; {\rm on} \; \partial M.
\end{equation}
Along $ \partial M $ the function $ u_{-} \equiv 0 $, hence above holds for all $ \zeta > 0 $.

Lastly we construct the super-solution and determine the upper bound of $ \zeta > 0 $. Choose $ \zeta > 0 $ satisfying
\begin{equation}\label{yamabe:eqn15}
\beta \inf_{\partial M} \phi \geqslant \zeta \frac{2}{p-2}\sup_{\partial M} \phi^{\frac{p}{2}} \Rightarrow \beta \phi \geqslant \zeta \frac{2}{p-2}\phi^{\frac{p}{2}} \; {\rm on} \; \partial M.
\end{equation}
Any even smaller $ \zeta $ are good. Pick up $ \gamma \ll 1 $ such that
\begin{equation}\label{yamabe:eqn16}
0 < 20 \lambda \gamma + 2\gamma \cdot \sup_{M} \lvert R_{g} \rvert \gamma < \frac{\theta}{2}, 31 \lambda (\phi + \gamma)^{4} \gamma < \frac{\theta}{2}. 
\end{equation}
Set
\begin{align*}
V & = \lbrace x \in \Omega: u_{1}(x) > \phi(x) \rbrace, V' = \lbrace x \in \Omega: u_{1}(x) < \phi(x) \rbrace, D = \lbrace x \in \Omega : u_{1}(x) = \phi(x) \rbrace, \\
D' & = \lbrace x \in \Omega : \lvert u_{1}(x) - \phi(x) \rvert < \gamma \rbrace, D'' = \left\lbrace x \in \Omega : \lvert u_{1}(x) - \phi(x) \rvert > \frac{\gamma}{2} \right\rbrace.
\end{align*}
If $ \phi \geqslant u_{1} $ pointwise, then $ \phi $ is a super-solution. If not, a good candidate of super-solution will be $ \max \lbrace u_{1}, \phi \rbrace $ in $ \Omega $ and $ \phi $ outside $ \Omega $, this is an $ H^{1} \cap \calC_{0} $-function. Let $ \nu $ be the outward normal derivative of $ \partial V $ along $ D $. If $ \frac{\partial u_{1}}{\partial \nu} = - \frac{\partial \phi}{\partial \nu} $ on $ D $ then the super-solution has been constructed. However, this is in general not the case. Define
\begin{equation}\label{yamabe:eqn17}
\Omega_{1} = V \cap D'', \Omega_{2} = V' \cap D'', \Omega_{3} = D'.
\end{equation}
Construct a specific smooth partition of unity $ \lbrace \chi_{i} \rbrace $ subordinate to $ \lbrace \Omega_{i} \rbrace $ as in Theorem 4.3 of \cite{XU3}, we define
\begin{equation}\label{yamabe:eqn18}
\bar{u} = \chi_{1} u_{1} + \chi_{2} \phi + \chi_{3} \left( \phi + \gamma \right).
\end{equation}
Without loss of generality, we may assume that all $ \Omega_{i}, i = 1, 2, 3 $ are connected. Due to the same argument in Theorem 4.3 of \cite{XU3}, we conclude that $ \bar{u} \in \calC^{\infty}(\Omega) $ is a super-solution of the perturbed bounded Yamabe equation in $ \Omega $ pointwise, regardless of the boundary condition at the time being. By the definition of $ \bar{u} $, it is immediate that $ \bar{u} \geqslant u_{1} $. Then define
\begin{equation}\label{yamabe:eqn19}
u_{+} : = \begin{cases} \bar{u}, & \; {\rm in} \; \Omega; \\ \phi, & \; {\rm in} \; M \backslash \Omega. \end{cases}
\end{equation}
It follows that $ u_{+} \in \calC^{\infty}(M) $ since $ \bar{u} = \phi $ near $ \partial \Omega $. Since $ u_{+} = \phi $ on $ \partial M $, we conclude from (\ref{yamabe:eqn19}) that
\begin{equation}\label{yamabe:eqn20}
-a\Delta_{g} u_{+} + \left( R_{g} + \tau \right) u_{+} \geqslant \lambda_{\tau} u_{+}^{p-1} \; {\rm in} \; M, \frac{\partial u_{+}}{\partial \nu} + \frac{2}{p-2} h_{g} u_{+} \geqslant \frac{2}{p-2} \zeta u_{+}^{\frac{p}{2}} \; {\rm on} \; \partial M.
\end{equation}
Furthermore $ 0 \leqslant u_{-} \leqslant u_{+} $ and $ u_{-} \not\equiv 0 $. Thus by Theorem \ref{iteration:cor1}, we conclude that there exists a real, positive function $ u \in \calC^{\infty}(M) \cap \calC^{1, \alpha}(M) $ that solves (\ref{yamabe:eqn1}).
\end{proof}
\begin{remark}\label{yamabe:re1}
Note that if (\ref{yamabe:eqns5}) and (\ref{yamabe:eqns6}) hold for some $ \tau_{0} $, then both hold for all $ \tau \in [\tau_{0}, 0] $.
\end{remark}
\medskip

As positive eigenvalue cases in \cite{XU3, XU4}, we take the limit $ \tau \rightarrow 0^{-} $. Compared with minimal boundary case in \cite{XU4}, the essential difficulty here is the nonlinear term $ u^{\frac{p}{2}} $ on $ \partial M $. Thus we also need to show the uniform boundedness of $ u_{\tau} $ in (\ref{yamabe:eqns1}) provided the existence in Theorem \ref{yamabe:thm3}. Pick some $ \tau_{0} < 0 $, $ \lvert \tau_{0} \rvert $ small enough, such that Theorem \ref{yamabe:thm3} holds. For all $ \tau \in [\tau_{0}, 0) $, we have
\begin{equation}\label{yamabe:eqn21}
-a\Delta_{g} u_{\tau} + \left(R_{g} + \tau \right) u_{\tau} = \lambda_{\tau} u_{\tau}^{p - 1} \; {\rm in} \; M, \frac{\partial u_{\tau}}{\partial \nu} + \frac{2}{p - 2} h_{g} u_{\tau} = \frac{2}{p-2} \zeta u_{\tau}^{\frac{p}{2}} \; {\rm on} \; \partial M
\end{equation}
for a fixed $ \zeta > 0 $, which depends on the choice of $ \beta $ in (\ref{yamabe:eqns3}) only. Due to the construction of sub-solutions and super-solutions of (\ref{yamabe:eqn21}) in Theorem \ref{yamabe:thm3}, each $ u_{\tau} $ is associated with $ 0 \leqslant u_{\tau, -} \leqslant u_{\tau} \leqslant u_{\tau, +} $ where
\begin{equation}\label{yamabe:eqn22}
u_{\tau, -}  = \begin{cases} \tilde{u}_{\tau}, & \text{within} \; \Omega \\ 0, & \text{outside} \; \Omega \end{cases};  u_{\tau, +} = \begin{cases}  \chi_{\tau ,1} \tilde{u}_{\tau} + \chi_{\tau, 2} \phi + \chi_{\tau, 3} \left( \phi + \gamma \right), & \; {\rm in} \; \Omega \\ \phi, & \; {\rm in} \; M \backslash \Omega \end{cases}.
\end{equation}
Here $ \tilde{u}_{\tau} $ are local solutions of the PDEs
\begin{equation}\label{yamabe:eqn23}
-a\Delta_{g} \tilde{u}_{\tau} + \left(R_{g} + \tau \right) \tilde{u}_{\tau} = \lambda_{\tau} \tilde{u}_{\tau}^{p - 1} \; {\rm in} \; \Omega, \tilde{u}_{\tau} = 0 \; {\rm on} \; \partial \Omega;
\end{equation}
in addition, $ 0 \leqslant \chi_{\tau, i} \leqslant 1, i = 1, 2, 3 $. Next proposition shows that $ \lbrace \tilde{u}_{\tau} \rbrace $ is uniformly bounded in $ \calC^{2, \alpha} $-sense, provided that $ \Omega $ is small enough due to Remark 2.1.
\begin{proposition}\label{yamabe:prop1}
Let $ \tau_{0} $ be a negative constant such that $ \lambda{\tau_{0}} > 0 $. Let $ (\Omega, g) $ be a small enough interior Riemannian domain of $ (M, \partial M, g) $ in which (\ref{yamabe:eqn23}) admits a positive solution by Proposition \ref{pre:prop4} for every $ \tau \in [\tau_{0}, 0) $. Then there exists a constant $ \mathcal{K} $ such that
\begin{equation}\label{yamabe:eqn24}
\lVert \tilde{u}_{\tau} \rVert_{\calC^{2, \alpha}(\Omega)} \leqslant \mathcal{K}, \forall \tau \in [\tau_{0}, 0).
\end{equation}
\end{proposition}
\begin{proof} Due to the hypothesis of $ \tau_{0} $, we have $ \lambda_{\tau_{0}} > 0 $. As what we did in Theorem 4.3 of \cite{XU3} and Theorem 5.6 of \cite{XU4}, we show that $ \lambda_{\tau} \in [\lambda_{\tau_{0}}, \lambda(\mathbb{S}_{+}^{n})] $ and
\begin{equation}\label{yamabe:eqn25}
\lVert \tilde{u}_{\tau} \rVert_{\calL^{p}(\Omega, g)} \leqslant C_{1}', \forall \tau \in [\tau_{0}, 0).
\end{equation}
We may assume $ \int_{M} d\omega = 1 $ for this continuity verification, since otherwise only an extra term with respect to $ \text{Vol}_{g} $ will appear. Recall from (\ref{yamabe:eqns2}) that
\begin{equation*}
\lambda_{\tau} = \inf_{u \neq 0, u \in H^{1}(M)} \left\lbrace \frac{\int_{M} a\lvert \nabla_{g} u \rvert^{2} d\omega + \int_{M} \left( R_{g} + \tau \right) u^{2} d\omega + \int_{\partial M} \frac{2a}{p-2}h_{g} u^{2} dS}{\left( \int_{M} u^{p} d\omega \right)^{\frac{2}{p}}} \right\rbrace.
\end{equation*}
It is immediate that if $ \tau_{1} < \tau_{2} < 0 $ then $ \lambda_{\tau_{1}} \leqslant \lambda_{\tau_{2}} $. For continuity we assume $ 0 < \tau_{2} - \tau_{1} < \gamma $. For each $ \epsilon > 0 $, there exists a function $ u_{0} $ such that
\begin{equation*}
\frac{\int_{M} a\lvert \nabla_{g} u_{0} \rvert^{2} d\omega + \int_{M} \left( R_{g} + \tau_{1} \right) u_{0}^{2} d\omega + \int_{\partial M} \frac{2a}{p-2}h_{g} u_{0}^{2} dS}{\left( \int_{M} u_{0}^{p} d\omega \right)^{\frac{2}{p}}} < \lambda_{\tau_{1}} + \epsilon.
\end{equation*}
It follows that
\begin{align*}
\lambda_{\tau_{2}} & \leqslant \frac{\int_{M} a\lvert \nabla_{g} u_{0} \rvert^{2} d\omega + \int_{M} \left( R_{g} + \tau_{2} \right) u_{0}^{2} d\omega + \int_{\partial M} \frac{2a}{p-2}h_{g} u_{0}^{2} dS}{\left( \int_{M} u_{0}^{p} d\omega \right)^{\frac{2}{p}}} \\
& \leqslant \frac{\int_{M} a\lvert \nabla_{g} u_{0} \rvert^{2} d\omega + \int_{M} \left( R_{g} + \tau_{1} \right) u_{0}^{2} d\omega + \int_{\partial M} \frac{2a}{p-2}h_{g} u_{0}^{2} dS}{\left( \int_{M} u_{0}^{p} d\omega \right)^{\frac{2}{p}}} + \frac{\left( \tau_{2} - \tau_{1} \right) \int_{M} u_{0}^{2} d\omega}{\left( \int_{M} u_{0}^{p} d\omega \right)^{\frac{2}{p}}} \\
& \leqslant \lambda_{\tau_{1}} + \epsilon + \tau_{2} - \tau_{1} < \lambda_{\beta_{1}} + \epsilon + \tau_{2} - \tau_{1}.
\end{align*}
Since $ \epsilon $ is arbitrarily small, we conclude that
\begin{equation*}
0 < \tau_{2} - \tau_{1} < \gamma \Rightarrow \lvert \lambda_{\tau_{2}} - \lambda_{\tau_{1}} \rvert \leqslant 2\gamma.
\end{equation*}
By equation (4) in \cite[\S1]{ESC} , we conclude that
\begin{equation*}
\lambda_{\beta} \leqslant \lambda(\mathbb{S}_{+}^{n})
\end{equation*}
Thus we conclude that
\begin{equation}\label{yamabe:eqn26}
\lambda_{\tau_{0}} \leqslant \lambda_{\tau} \leqslant \lambda(\mathbb{S}_{+}^{n}), \forall \tau \in [\tau_{0}, 0), \lim_{\tau \rightarrow 0^{-}} \lambda_{\tau} : = \lambda.
\end{equation}
Next we show that (\ref{yamabe:eqn25}) holds for all $ \tau \in [\tau_{0}, 0] $. Pairing $ \tilde{u}_{\beta} $ on both sides of (\ref{yamabe:eqn23}),
\begin{equation}\label{yamabe:eqn27}
\begin{split}
a\lVert \nabla_{g} \tilde{u}_{\tau} \rVert_{\calL^{2}(\Omega, g)}^{2} & =  \lambda_{\tau} \lVert \tilde{u}_{\tau} \rVert_{\calL^{p}(\Omega, g)}^{p} - \int_{M} \left( R_{g} + \tau \right) u_{\tau}^{2} \dvol \\
\Rightarrow \lambda_{\tau} \lVert \tilde{u}_{\tau} \rVert_{\calL^{p}(\Omega, g)}^{p} & \leqslant a\lVert \nabla_{g} \tilde{u}_{\tau} \rVert_{\calL^{2}(\Omega, g)}^{2} + \left( \sup_{M} \lvert R_{g} \rvert + \lvert \tau \rvert \right) \lVert u_{\tau} \rVert_{\calL^{2}(\Omega, g)}^{2}.
\end{split}
\end{equation}
Recall the functional
\begin{equation*}
J(u) = \int_{\Omega} \left( \frac{1}{2} \sum_{i, j} a_{ij}(x) \partial_{i}u \partial_{j} u - \frac{\lambda_{\tau} \sqrt{\det(g)} }{p} u^{p} - \frac{1}{2} \left(R_{g} + \tau \right) u^{2} \sqrt{\det(g)} \right) d
\end{equation*}
and the constant $ K_{0} $ in \cite[\S3]{XU3}. Note that $ K_{0} $ only depends on $ \lambda_{\tau} $, hence $ K_{0} $ is uniformly bounded above when $ \tau \in [\tau_{0}, 0) $. We denote this upper bound by $ K_{0} $ again. Due to Theorem 1.1 of \cite{WANG}, each solution $ \tilde{u}_{\tau} $ satisfies
\begin{equation}\label{yamabe:eqn28}
J(\tilde{u}_{\tau}) \leqslant K_{0} \Rightarrow \frac{a}{2} \lVert \nabla_{g} \tilde{u}_{\tau} \rVert_{\calL^{2}(\Omega, g)}^{2} - \frac{\lambda_{\tau}}{p} \lVert \tilde{u}_{\tau} \rVert_{\calL^{p}(\Omega, g)}^{p} - \frac{1}{2} \int_{M} \left(R_{g} + \tau \right) \tilde{u}_{\tau}^{2} \dvol \leqslant K_{0}.
\end{equation}
Let $ \lambda_{1} $ be the first eigenvalue of $ -\Delta_{g} $ on $ \Omega $ with Dirichlet boundary condition. Apply the estimate (\ref{yamabe:eqn27}) into (\ref{yamabe:eqn28}), we have
\begin{align*}
\frac{a}{2} \lVert \nabla_{g} \tilde{u}_{\tau} \rVert_{\calL^{2}(\Omega, g)}^{2} & \leqslant K_{0} + \frac{1}{p} \left(a\lVert \nabla_{g} \tilde{u}_{\tau} \rVert_{\calL^{2}(\Omega, g)}^{2} + \left( \sup_{M} \lvert R_{g} \rvert + \lvert \tau \rvert \right) \lVert u_{\tau} \rVert_{\calL^{2}(\Omega, g)}^{2} \right) \\
& \qquad + \frac{1}{2} \left( \sup_{M} \lvert R_{g} \rvert + \lvert \tau \rvert \right) \lVert u_{\tau} \rVert_{\calL^{2}(\Omega, g)}^{2} \\
& \leqslant K_{0} + \frac{a(n - 2)}{2n} \lVert \nabla_{g} \tilde{u}_{\tau} \rVert_{\calL^{2}(\Omega, g)}^{2} \\
& \qquad + \left( \frac{n - 2}{2n} + \frac{1}{2} \right) \left( \sup_{M} \lvert R_{g} \rvert + \lvert \tau \rvert \right) \cdot \lambda_{1}^{-1} \lVert \nabla_{g} \tilde{u}_{\tau} \rVert_{\calL^{2}(\Omega, g)}^{2}; \\
\Rightarrow & \left(\frac{a}{n} - \left( \frac{n - 2}{2n} + \frac{1}{2} \right) \left( \sup_{M} \lvert R_{g} \rvert + \lvert \tau \rvert \right) \cdot \lambda_{1}^{-1} \right) \lVert \nabla_{g} \tilde{u}_{\tau} \rVert_{\calL^{2}(\Omega, g)}^{2} \leqslant K_{0}.
\end{align*}
Recall in Remark 2.1 that $ \Omega $ is small enough so that 
\begin{equation*}
\frac{a}{n} - \left( \frac{n - 2}{2n} + \frac{1}{2} \right) \left( \sup_{M} \lvert R_{g} \rvert + \lvert \tau_{0} \rvert \right) \cdot \lambda_{1}^{-1} > 0
\end{equation*}
holds for $ \tau_{0} $, thus for all $ \tau_{0} \in [\tau_{0}, 0] $. It follows from above that there exists a constant $ C_{0}' $ such that
\begin{equation*}
\lVert \nabla_{g} \tilde{u}_{\tau} \rVert_{\calL^{2}(\Omega, g)}^{2} \leqslant C_{0}', \forall \tau \in [\tau_{0}, 0].
\end{equation*}
Apply (\ref{yamabe:eqn27}) with the other way around, we conclude that
\begin{align*}
\lambda_{\tau} \lVert \tilde{u}_{\tau} \rVert_{\calL^{p}(\Omega, g)}^{p} & \leqslant a\lVert \nabla_{g} \tilde{u}_{\tau} \rVert_{\calL^{2}(\Omega, g)}^{2} + \left( \sup_{M} \lvert R_{g} \rvert + \lvert \tau \rvert \right)  \lVert \tilde{u}_{\tau} \rVert_{\calL^{2}(\Omega, g)}^{2} \\
& \leqslant \left( a + \left( \sup_{M} \lvert R_{g} \rvert + \lvert \tau \rvert \right)  \lambda_{1}^{-1} \right) \lVert \nabla_{g} u_{\tau} \rVert_{\calL^{2}(\Omega, g)}^{2}.
\end{align*}
Furthermore, it follows from the characterization of $ \lambda_{\tau} $ and observe that the small domain $ \Omega $ we chose is an interior domain hence $ u_{\tau} \equiv 0 $ on $ \partial M $. Applying (\ref{yamabe:eqn27}), we have
\begin{equation*}
\lambda_{\tau} \leqslant \frac{\int_{\Omega} a \lvert \nabla_{g} \tilde{u}_{\tau} \rvert^{2} \dvol + \int_{\Omega} \left( R_{g} + \tau \right) \tilde{u}_{\tau}^{2} \dvol}{\left( \int_{\Omega} \tilde{u}_{\tau}^{p} d\omega \right)^{\frac{2}{p}}} = \frac{\lambda_{\tau} \lVert \tilde{u}_{\tau} \rVert_{\calL^{p}(\Omega, g)}^{p}}{\lVert \tilde{u}_{\tau} \rVert_{\calL^{p}(\Omega, g)}^{\frac{2}{p}}}
\end{equation*}
We conclude that
\begin{equation*}
1 \leqslant \lVert \tilde{u}_{\tau} \rVert_{\calL^{p}(\Omega, g)} \leqslant C_{1}', \forall \tau \in [\tau_{0}, 0).
\end{equation*}
Note that this uniform upper bound $ C_{1}' $ is unchanged if we further shrink the domain $ \Omega $. Note that this shrinkage of domain is a restriction, not a scaling of domain or metric. We can, without loss of generality, assume that $ C_{1}' = 1 $. This can be done by scaling the metric one time, uniformly for all $ \tau \in [\tau_{0}, 0) $. Note that this scaling does not affect the local solvability in Proposition \ref{pre:prop4} since the estimates in Appendix A of \cite{XU3} still hold under scaling. After one-time scaling $ g \mapsto \delta g $ we still have $ \lambda_{\delta^{2}\tau} \in [\lambda_{\delta^{2}\tau_{0}}, \lambda(\mathbb{S}_{+}^{n})] $ due to the characterization of $ \lambda_{\tau} $ and scaling invariance of $ \lambda(\mathbb{S}_{+}^{n}) $. Since $ \tau < 0 $, the lower bound of $ \lambda_{\delta^{2}\tau_{0}} $ is unchanged if $ \delta < 1 $. In addition, we also have the scaling invariance of the sharp Sobolev inequality with critical exponent that will be used below. We still denote the new metric by $ g $, which follows that
\begin{equation}\label{yamabe:eqn29}
\mathcal{K}_{3}' \leqslant \lVert \tilde{u}_{\tau} \rVert_{\calL^{p}(\Omega, g)}^{p} \leqslant 1, \forall \tau \in [\tau_{0}, 0).
\end{equation}
According to equation (4) of \cite[\S1]{ESC}, we have
\begin{equation}\label{yamabe:eqn30}
\lambda_{\tau} \leqslant \lambda(\mathbb{S}_{+}^{n}) = \frac{n(n - 2)}{4} \text{Vol}\left(\mathbb{S}_{+}^{n}\right)^{\frac{2}{n}} = 2^{-\frac{2}{n}} \frac{n(n - 2)}{4} \text{Vol}\left(\mathbb{S}^{n}\right)^{\frac{2}{n}} = 2^{-\frac{2}{n}} \lambda(\mathbb{S}^{n}) = 2^{-\frac{2}{n}} aT.
\end{equation}
It follows from (\ref{yamabe:eqn30}) that the ratio $ \frac{\lambda_{\tau}}{aT} < 1 $ holds after one time scaling in the sense of (\ref{yamabe:eqn29}). Due to the idea of Trudinger and Aubin, the argument in Theorem 4.4 of \cite{XU3} and Theorem 5.6 of \cite{XU4}, we pair $ \tilde{u}_{\tau}^{1 + 2 \delta} $ for some $ \delta > 0 $ on both sides of (\ref{yamabe:eqn23}) and denote $ w_{\tau} = \tilde{u}_{\tau}^{1 + \delta} $, we have
\begin{align*}
& \int_{\Omega} a \nabla_{g} \tilde{u}_{\tau} \cdot \nabla_{g} \left(\tilde{u}_{\tau}^{1 + 2\delta} \right) \dvol + \int_{\Omega} \left(R_{g} + \tau \right) u_{\beta}^{2 + 2\delta} \dvol = \lambda_{\tau} \int_{\Omega} \tilde{u}_{\tau}^{p + 2\delta} \dvol; \\
\Rightarrow & \frac{1 + 2\delta}{1 + \delta^{2}} \int_{\Omega} a \lvert \nabla_{g} w_{\tau} \rvert^{2} \dvol = \lambda_{\tau} \int_{\Omega} w_{\tau}^{2} \tilde{u}_{\tau}^{p-2} \dvol - \int_{\Omega} \left(R_{g} + \tau \right) w_{\tau}^{2} \dvol.
\end{align*}
When the radius $ r $ of $ \Omega $ is small enough, there exists a constant $ A $ such that
\begin{equation*}
\lVert u \rVert_{\calL^{p}(\Omega, g)}^{2} \leqslant (1 + Ar^{2}) \lVert u \rVert_{\calL^{p}(\Omega)}^{2}, \lVert D u \rVert_{\calL^{2}(\Omega)}^{2} \leqslant (1 + Ar^{2}) \lVert \nabla_{g} u \rVert_{\calL^{2}(\Omega, g)}^{2}.
\end{equation*}
Due to standard sharp Sobolev embedding on Euclidean space, we have
\begin{align*}
\lVert w_{\tau} \rVert_{\calL^{p}(\Omega, g)}^{2} & \leqslant (1 + Ar^{2}) \lVert w_{\tau} \rVert_{\calL^{p}(\Omega)}^{2} \leqslant \frac{1 + Ar^{2}}{T} \lVert D w_{\tau} \rVert_{\calL^{2}(\Omega)}^{2} \leqslant \frac{\left( 1 + Ar^{2} \right)^{2}}{T} \lVert \nabla_{g} w_{\tau} \rVert_{\calL^{2}(\Omega, g)}^{2} \\
& = \frac{\left( 1 + Ar^{2} \right)^{2}}{aT} \cdot \frac{1 + \delta^{2}}{1 + 2\delta} \left(\lambda_{\tau} \int_{\Omega} w_{\tau}^{2} \tilde{u}_{\tau}^{p-2} \dvol - \int_{\Omega} \left(R_{g} + \tau \right) w_{\tau}^{2} \dvol \right) \\
& \leqslant \frac{\left( 1 + Ar^{2} \right)^{2}}{aT} \cdot \frac{1 + \delta^{2}}{1 + 2\delta} \lambda_{\tau} \lVert w_{\tau} \rVert_{\calL^{p}(\Omega, g)}^{2} \lVert \tilde{u}_{\tau} \rVert_{\calL^{p}(\Omega, g)}^{p-2} + C_{\tau} \lVert w_{\tau} \rVert_{\calL^{2}(\Omega, g)}^{2} \\
& \leqslant \left( 1 + Ar^{2} \right)^{2} \cdot \frac{1 + \delta^{2}}{1 + 2\delta} \cdot \frac{2^{-\frac{2}{n}} aT}{aT} \lVert w_{\tau} \rVert_{\calL^{p}(\Omega, g)}^{2} + C_{\tau} \lVert w_{\tau} \rVert_{\calL^{2}(\Omega, g)}^{2}
\end{align*}
by H\"older's inequality and (\ref{yamabe:eqn30}). Note that $ C_{\tau} $ is uniformly bounded above for all $ \tau \in [\tau_{0}, 0) $. Due to the last line above, we can choose $ r, \delta $ small enough so that
\begin{equation*}
\left( 1 + Ar^{2} \right)^{2} \cdot \frac{1 + \delta^{2}}{1 + 2\delta} \cdot \frac{2^{-\frac{2}{n}} aT}{aT} < 1.
\end{equation*}
It follows that
\begin{equation*}
\lVert w_{\tau} \rVert_{\calL^{p}(\Omega, g)}^{2} \leqslant \mathcal{K}_{1} \lVert w_{\tau} \rVert_{\calL^{2}(\Omega, g)}^{2}.
\end{equation*}
Recall that $ w_{\tau} = \tilde{u}_{\tau}^{1 + \delta} $. Applying H\"older's inequality on right side above, and note that $ \text{Vol}_{g}(\Omega) \leqslant \text{Vol}_{g}(M) $, we conclude by exactly the same argument as in \cite[Prop.~4.4]{PL}, \cite[Thm.~4.4]{XU3} that
\begin{equation}\label{yamabe:eqn31}
\lVert \tilde{u}_{\tau} \rVert_{\calL^{r}(\Omega, g)} \leqslant \mathcal{K}_{2}, r = p(1 + \delta), \forall \tau \in [\tau_{0}, 0).
\end{equation}
Applying a standard bootstrapping method with elliptic regularity and Sobolev embedding, it follows from (\ref{yamabe:eqn31}) that (\ref{yamabe:eqn24}) holds.
\end{proof}
\begin{remark}\label{yamabe:sre2}
The key theorem \cite[Theorem A]{PL} given by Yamabe, Trudinger and Aubin showed that the Yamabe equation has a desired solution if the Yamabe invariant of some closed manifold is strictly less than a threshold. They applied a one-time scaling of the metric in their proof, in which they normalize the volume of the manifold. Here we also apply a one-time scaling of the metric. The reason that we can apply the scaling here is due to the nonlinear term $ u^{p-1} $. We do not need to pass the limit from sub-critical exponent case to critical exponent case.
\end{remark}
\medskip

An immediate consequence of Proposition (\ref{yamabe:prop1}) is that $ \lVert u_{\tau, +} \rVert_{\calL^{r}(M, g)} $ are uniformly bounded for all $ \tau \in [\tau_{0}, 0) $. So are $ \lVert u_{\tau} \rVert_{\calL^{r}(M, g)} $ with $ r = p( 1 + \delta) > p $. Now we can show the existence of the positive solution of the boundary Yamabe equation when $ \eta_{1} > 0 $.
\begin{theorem}\label{yamabe:thm4}
Let $ (M, \partial M, g) $ be a compact manifold with non-empty smooth boundary. Assume $ R_{g} < 0 $ somewhere in $ M $ and $ h_{g} > 0 $ everywhere on $ \partial M $. When $ \eta_{1} > 0 $, there exists some $ \lambda > 0 $ and $ \zeta > 0 $ such that the boundary Yamabe equation (\ref{yamabe:eqn1}) has a real, positive solution $ u \in \calC^{\infty}(M) $.
\end{theorem}
\begin{proof} Pick up some $ \tau_{0} < 0 $ and $ \lvert \tau_{0} \rvert $ small enough. The following PDE
\begin{equation*}
-a\Delta_{g} u_{\tau_{0}} + \left(R_{g} + \tau_{0} \right) u_{\tau_{0}} = \lambda_{\tau_{0}} u_{\tau_{0}}^{p - 1} \; {\rm in} \; M, \frac{\partial u_{\tau_{0}}}{\partial \nu} + \frac{2}{p - 2} h_{g} u_{\tau_{0}} = \frac{2}{p-2} \zeta u_{\tau_{0}}^{\frac{p}{2}} \; {\rm on} \; \partial M
\end{equation*}
admits a real, positive solution $ u_{\tau_{0}} \in \calC^{\infty}(M) $. Due to the discussion above, there exists a sequence of smooth functions $ \lbrace u_{\tau} \rbrace, \tau \in [\tau_{0}. 0) $, each $ u_{\tau} $ is a real, positive solution of 
\begin{equation}\label{yamabe:eqn32a}
-a\Delta_{g} u_{\tau} + \left(R_{g} + \tau \right) u_{\tau} = \lambda_{\tau} u_{\tau}^{p - 1} \; {\rm in} \; M, \frac{\partial u_{\tau}}{\partial \nu} + \frac{2}{p - 2} h_{g} u_{\tau} = \frac{2}{p-2} \zeta u_{\tau}^{\frac{p}{2}} \; {\rm on} \; \partial M,
\end{equation}
respectively. Recall that when we construct super-solutions in Theorem \ref{yamabe:thm3}, we determine the eigenfunction $ \phi $ by letting
\begin{equation*}
\left( \eta_{1, \beta} + \tau \right) \inf_{M} (\delta \varphi) > 2^{p-2} \lambda_{\tau} \sup_{M} \left(\delta^{p-1} \varphi^{p-1} \right)
\end{equation*}
Since both $ \tau $ and $ \lambda_{\tau} $ are bounded above and below, a uniform choice of $ \delta $ is good enough for all $ \tau \in [\tau_{0}, 0) $. Due to (\ref{yamabe:eqn22}) and Proposition \ref{yamabe:prop1}, we conclude that
\begin{equation}\label{yamabe:eqn32}
\lVert u_{\tau} \rVert_{\calL^{r}(M, g)} \leqslant \mathcal{K}_{0}', \forall \tau \in [\tau_{0}, 0), r > p.
\end{equation}
By repeated elliptic regularities and Sobolev embedding, (\ref{yamabe:eqn32}) implies that
\begin{equation}\label{yamabe:eqn33}
\lVert u_{\tau} \rVert_{\calC^{2, \alpha}(M)} \leqslant \mathcal{K}_{0}, \forall \tau \in [\tau_{0}, 0).
\end{equation}
By Arzela-Ascoli, it follows that up to a subsequence, $ \lim_{\tau \rightarrow 0^{-}} u_{\tau} = u $. Due to (\ref{yamabe:eqn26}), we have $ \lim_{\tau \rightarrow 0^{-}} \lambda_{\tau} = \lambda $. It follows that the limiting function $ u $ satisfies
\begin{align*}
-a\Delta_{g} u + R_{g} u = \lambda u^{p-1} \; {\rm in} \; M; \\
\frac{\partial u}{\partial \nu} + \frac{2}{p-2} h_{g} u = \frac{2}{p-2} \zeta u^{\frac{p}{2}} \; {\rm on} \; \partial M.
\end{align*}
By \cite{Che} we conclude that $ u \in \calC^{\infty}(M) $. Lastly we show $ u > 0 $. Due to the construction of $ u_{\tau} $, it suffices to show $ \tilde{u}_{\tau} $ in (\ref{yamabe:eqn25}) are uniformly bounded below. Due to (\ref{yamabe:eqn29}) in Proposition \ref{yamabe:prop1}, we conclude that $ \lVert \tilde{u}_{\tau} \rVert_{\calL^{p}(\Omega, g)} \geqslant \mathcal{K}_{3}', \forall \tau \in [\tau_{0}, 0) $. Since $ u_{\tau} \geqslant \tilde{u}_{\tau} \geqslant 0 $ pointwise, it follows that
\begin{equation*}
\lVert u_{\tau} \rVert_{\calL^{p}(M, g)} \geqslant \mathcal{K}_{3} > 0, \forall \tau \in [\tau_{0}, 0).
\end{equation*}
As a consequence of Arzela-Ascoli theorem, it follows that
\begin{equation*}
\lVert u \rVert_{\calL^{p}(M, g)} > 0.
\end{equation*}
It then follows from maximum principle and \cite[\S1]{ESC} that $ u > 0 $ on $ M $.
\end{proof}
\medskip

\begin{remark}\label{yamabe:re2}
The key steps for positive eigenvalue case are: (i) a local solution of Yamabe equation with Dirichlet condition, which plays a central role as a nontrivial sub-solution; the existence of such a solution is proven in Appendix A of \cite{XU3}; (ii) a particular choices of partition of unity in constructing a super-solution of Yamabe equation in $ \Omega $, this specific construction relies on the solutions of linear PDEs, the details can be found in the proof of Theorem 4.3 of \cite{XU3}.

We point out again that the existence of a local solution of Yamabe equation and the construction of a local super-solution can help us avoid the use of Weyl tensor. We do not need to classify the boundary points also, since the local solution is within a subset of interior of the manifold, and the sub-solution we constructed is trivial on $ \partial M $. See also Remark \ref{pre:re00}
\end{remark}

\bibliographystyle{plain}
\bibliography{Yamabessg}

\end{document}